\documentclass[11pt]{article}
\usepackage[english]{babel}
\usepackage{amsmath,amsthm}
\usepackage{amsfonts}
\usepackage{mathrsfs}
\usepackage{color}
\usepackage{bbm}
\usepackage{enumerate}
\usepackage{amsrefs}
\usepackage{mathtools}
\usepackage{hyperref}

\newtheorem{theorem}{Theorem}[section]
\newtheorem{corollary}[theorem]{Corollary}
\newtheorem{lemma}[theorem]{Lemma}
\newtheorem{proposition}[theorem]{Proposition}
\newtheorem{assumption}{Assumption}
\theoremstyle{definition}
\newtheorem{definition}[theorem]{Definition}
\theoremstyle{remark}
\newtheorem{remark}[theorem]{Remark}

\numberwithin{equation}{section}
\begin{document}
\def\Pro{{\mathbb{P}}}
\def\E{{\mathbb{E}}}
\def\e{{\varepsilon}}
\def\veps{{\varepsilon}}
\def\ds{{\displaystyle}}
\def\nat{{\mathbb{N}}}
\def\Dom{{\textnormal{Dom}}}
\def\dist{{\textnormal{dist}}}
\def\R{{\mathbb{R}}}
\def\O{{\mathcal{O}}}
\def\T{{\mathcal{T}}}
\def\Tr{{\textnormal{Tr}}}
\def\sgn{{\textnormal{sign}}}
\def\I{{\mathcal{I}}}
\def\A{{\mathcal{A}}}
\def\H{{\mathcal{H}}}
\def\S{{\mathcal{S}}}

\title{Systems of small-noise stochastic reaction-diffusion equations satisfy a large deviations principle that is uniform over all initial data}%
\author{M. Salins\\ Boston University \\ msalins@bu.edu}
\maketitle

\begin{abstract}
Large deviations principles characterize the exponential decay rates of the probabilities of rare events.
Cerrai and R\"ockner \cite{cr-2004} proved that systems of stochastic reaction-diffusion equations satisfy a large deviations principle that is uniform over bounded sets of initial data.

This paper proves uniform large deviations results for a system of stochastic reaction--diffusion equations in a more general setting than Cerrai and R\"ockner. Furthermore, this paper identifies two common situations where the large deviations principle is uniform over unbounded sets of initial data, enabling the characterization of Freidlin-Wentzell exit time and exit shape asymptotics from unbounded sets.

%
%
%
\end{abstract}

%
%

\section{Introduction} \label{S:intro}

This paper investigates uniform large deviations principles for systems of stochastic reaction-diffusion equations. Let $\mathcal{O} \subset \mathbb{R}^d$ be a bounded open set with smooth boundary. Let $r \in \mathbb{N}$ be fixed. For $\e>0$, and continuous initial data $x:\mathcal{O} \times \{1,...,r\} \to \mathbb{R}$,  $X^\e_x(t,\xi) = (X^\e_{x,1}(t,\xi),..., X^\e_{x,r}(t,\xi))$ is the $\mathbb{R}^r$-valued random field solution to the equations for $i \in \{1,...,r\}$,

\begin{equation} \label{eq:intro-SPDE}
 \begin{cases}
  \displaystyle{\frac{\partial X^\e_{x,i}}{\partial t}(t,\xi) = \mathcal{A}_i X^\e_{x,i}(t,\xi) + f_i(t,\xi,X^\e_x(t,\xi))} \\
  \hspace{3cm}
  \displaystyle{+ \sqrt{\e}\sum_{n=1}^r\sigma_{in}(t,\xi,X^\e_x(t,\xi))\frac{\partial w_n}{\partial t}(t,\xi),}\\
  X^\e_x(t,\xi) = 0, \ \ \  \xi \in \partial \mathcal{O}, \ \ \  t \geq 0\\
  X^\e_x(0,\xi) = x(\xi), \ \ \  \xi \in \mathcal{O}.
 \end{cases}
\end{equation}

In the above equation, $\{\mathcal{A}_i\}_{i=1}^r$ are elliptic second-order differential operators, $\sigma_{in}$ are globally Lipschitz continuous in the third variable, and $f_i$ can be written as $f_i(t,\xi,x) = g_i(t,\xi,x_i) + h_i(t,\xi,x)$ where $g_i$ is continuous and non-increasing in its third argument and $h_i$ is globally Lipschitz continuous in its third argument. The multiplicative noise coefficients $\sigma_{in}$ are Lipschitz continuous in their third variable. The Gaussian noises $\frac{ \partial w_n}{\partial t}$, defined on some probability space $(\Omega, \mathcal{F},\Pro)$ are white in time, but possibly correlated in space.

As $\e \to 0$, the stochastic perturbations disappear and the solutions converges to the solution of the unperturbed system of partial differential equations $X^0_x$. This convergence is uniform over finite time intervals in the sense that for any $T>0$,
\begin{equation}
  \lim_{\e \to 0} \sup_{t \in [0,T]}\sup_{\xi \in \mathcal{O}}\sup_{i \in \{1,...,r\}}|X^\e_{x,i}(t,\xi) - X^0_{x,i}(t,\xi)| = 0 \text{ in probability}.
\end{equation}
Over unbounded time intervals, however, the stochastic system $X^\e_x$ will behave fundamentally differently than $X^0_x$ for any positive $\e>0$ if the $\sigma_{in}$ terms are sufficiently non-degenerate. For example, let $D \subset C(\mathcal{O}\times \{1,...,r\})$ be a collection of continuous functions that are invariant under the unperturbed dynamics. This means that if $x \in D$, then $X^0_x(t,\cdot) \in D$ for all $t>0$. Under reasonable assumptions on $D$ and $\sigma_{in}$, one can show that for any $\e>0$, $X^\e_x$ exits $D$ with probability one. The Freidlin-Wentzell exit time problem characterizes the exponential divergence rate of the exit time
\begin{equation} \label{eq:exit-time}
  \tau^\e_x : = \inf\{ t>0 : X^\e_x(t,\cdot) \not \in D\}
\end{equation}
as well as the limiting behavior of the exit shape $X^\e(\tau^\e_x,\cdot)$. Some results on exit time problems for stochastic partial differential equations can be found in \cite{f-1988,bcf-2015,c-1999,cs-2014-smolu-grad,cs-2016-smolu-AoP,cm-1997,cm-1990,dz,fjl-1982,sbd-2019}.

A large deviations principle characterizes the exponential decay rates of rare probabilities \cite{fw,dz,fk,v}. An important step for characterizing Freidlin-Wentzell exit behaviors is to prove that solutions $X^\e_x$ to the system \eqref{eq:intro-SPDE} satisfy a large deviations principle that is uniform with respect to the initial data $x \in D$.  The exact definition of the uniform large deviations principle is Definition \ref{def:ULDP}.

In \cite{cr-2004}, Cerrai and R\"ockner proved that systems of stochastic reaction-diffusion equations with a small non-Gaussian noisy perturbation, like \eqref{eq:intro-SPDE}, satisfy a uniform large deviations principle that is uniform over subsets of initial data that are bounded in the supremum norm. Their result significantly improved upon large deviations results by Freidlin \cite{f-1988}, Sowers \cite{s-1992}, Peszat \cite{p-1994}, and Kallianpur and Xiong \cite{kx-1996} by removing assumptions about global Lipschitz continuity of the reaction terms and ellipticity assumptions about the multiplicative noise terms. Furthermore, \cite{cr-2004} was the first paper to address the uniformity of the large deviations for the stochastic reaction-diffusion equation with respect to initial data in non-compact  bounded sets.

This current paper significantly strengthens Cerrai and R\"ockner's results. First, we further relax Cerrai and R\"ockner's assumptions, removing their assumptions about the local Lipschitz continuity and polynomial growth rate of the reaction terms. We assume only that the reaction terms that can be written as the sum of a decreasing function and a Lipschitz continuous function (see Assumption \ref{assum:vector-field}). 
We prove that this large class of stochastic reaction-diffusion equations satisfy a large deviations principle that is uniform over sets of initial data that are bounded in the supremum norm (Theorem \ref{thm:ULDP-bounded-subsets}).

The other main results of this paper show that in two common situations the large deviations principle holds uniformly over \textit{all} continuous initial data, not just over bounded subsets of initial data. These results enable the characterization of Freidlin-Wentzell exit time and exit shape asymptotics when the exit set $D$ is unbounded. 

If the multiplicative noise coefficients $\sigma_{in}$ are uniformly bounded, then the large deviations principle will hold uniformly over all continuous initial data (Theorem \ref{thm:ULDP-sigma-bounded}). In particular, whenever the system is exposed to additive noise, the large deviations principle holds uniformly over all initial data. Results about large deviations principles holding uniformly for unbounded sets of initial conditions, even in the additive noise case, were only previously known for equations with globally Lipschitz continuous reaction term \cite{s-2019}.

Next, we consider the case where the reaction terms $f_i$ feature super-linear dissipativity. This means that there exist constants $\mu>0$, $m>1$, and $c_0>0$ such that the decreasing functions $v_i \mapsto g_i(t,\xi,v_i)$ satisfy
\begin{equation}
  g_i(t,\xi,v_i)\sgn(v_i) \leq -\mu|v_i|^m \text{ for } |v_i|>c_0.
\end{equation}
This case is motivated by the Allen-Cahn equation where the reaction term is an odd-degree polynomial with negative leading coefficient such as $f_i(t,\xi,x) = -x_i^3 + x_i$ (see, for example, \cite{fjl-1982}). Such a superlinear dissipative nonlinearity strongly forces the solutions towards finite values. In the presence of super-linear dissipativity, the large deviations principle can be uniform over all continuous initial conditions even when the multiplicative noise coefficients are unbounded (Theorem \ref{thm:ULDP-super-dissip}). The allowable growth rate of the multiplicative noise coefficients depends on the degree $m$ of super-linear dissipativity of the reaction term. This is the first result showing that the large deviations principle holds uniformly over all initial conditions when the multiplicative noise coefficients are unbounded.

Proving that these large deviations principles hold uniformly over unbounded sets of initial data requires a fundamentally different approach than the one used by Cerrai and R\"ockner \cite{cr-2004}. Their argument relies on the assumption that the forcing term $f$ is locally Lipschitz continuous, and then uses localization techniques to approximate their equation by equations with globally Lipschitz forcing terms to prove the results. Such an approach can never lead to results that are uniform over unbounded sets of initial data.

The assumption of local Lipschitz continuity is replaced by a monotonicity condition (see, for example, \cite{mr-2010}). In Assumption \ref{assum:vector-field}, below, we assume that $f_i = g_i + h_i$ is the sum of a decreasing function $g_i$, which does not need to be locally Lipschitz continuous, and a globally Lipschitz continuous function $h_i$. In Section \ref{S:M} we prove that this assumption implies that an associated solution mapping is globally Lipschitz continuous, even when $g_i$ fails to be locally Lipschitz continuous. No localization techniques are required, enabling the proof of results that are uniform over unbounded subsets of initial data.

The proofs of the three main uniform large deviations results (Theorems \ref{thm:ULDP-bounded-subsets}, \ref{thm:ULDP-sigma-bounded}, and \ref{thm:ULDP-super-dissip}) are based on a variational principle for functions of infinite dimensional Wiener processes that is due to Budhiraja and Dupuis \cite{bd-2000}. In the context of the reaction-diffusion equation \eqref{eq:intro-SPDE}, Budhiraja and Dupuis proved that for any bounded, continuous $h: C([0,T]\times \bar{\mathcal{O}}\times \{1,...,r\}) \to \mathbb{R}$, $\e>0$, and $x \in C(\bar{\mathcal{O}}\times \{1,...,r\})$,
\begin{align}
  &\e \log \E \exp \left(-\frac{h \left( X^\e_x \right)}{\e}\right) 
  = -\inf_{u \in \mathscr{A}} \E \left[\frac{1}{2}\sum_{n=1}^r \int_0^T \int_{\mathcal{O}} |u_n(s,\xi)|^2 d\xi ds  + h \left( X^{\e,u}_x\right)\right].
\end{align}
In the above expression, $\mathscr{A}$ is a collection of stochastic controls $u \in L^2(\Omega \times [0,T]\times \mathcal{O} \times \{1,..,r\})$ that are adapted to the natural filtration of the driving noises and $X^{\e,u}_x$ is the solution to the controlled reaction diffusion equation
\begin{equation} \label{eq:intro-contolled-SPDE}
 \begin{cases}
  \displaystyle{\frac{\partial X^{\e,u}_{x,i}}{\partial t}(t,\xi) = \mathcal{A}_i X^{\e,u}_{x,i}(t,\xi) + f_i(t,\xi,X^{\e,u}_x(t,\xi))} \\
  \hspace{3cm}\displaystyle{+ \sqrt{\e}\sum_{n=1}^r\sigma_{in}(t,\xi,X^{\e,u}_x(t,\xi))\frac{\partial w_n}{\partial t}(t,\xi)}\\
   \hspace{3cm}\displaystyle{+ \sum_{n=1}^r\sigma_{in}(t,\xi,X^{\e,u}_x(t,\xi))Q_n u_n(t,\xi),}\\
  X^{\e,u}_x(t,\xi) = 0, \ \ \  \xi \in \partial \mathcal{O}, \ \ \  t \geq 0\\
  X^{\e,u}_x(0,\xi) = x(\xi), \ \ \  \xi \in \mathcal{O}.
 \end{cases}
\end{equation}
The linear operators $Q_n$ are the covariances of the noises $w_n$ (see Assumption \ref{assum:noise} below).

A major advancement in streamlining the proofs of uniform large deviations principles for small-noise SPDEs is the \textit{weak convergence approach} due to Budhiraja, Dupuis, and Maroulas \cite{bdm-2008}.
In the context of these reaction diffusion equations, their result shows that $X^\e_x$ satisfy a large deviations principle that holds uniformly over compact sets of initial data if whenever $x_n \to x$ in the supremum norm, $\e_n \to 0$ and $u_n \rightharpoonup u$ in distribution in the weak topology on $L^2([0,T]\times \mathcal{O} \times \{1,...,r\})$, the associated control problems $X^{\e_n,u_n}_{x_n}$ converge weakly to $X^{0,u}_x$.
Many authors have applied this approach to prove that many SPDEs satisfy large deviations results that are uniform over compact sets of initial data \cite{bm-2009,bm-2012,bb-2011,bdf-2012,bdm-2010,c-1999,cr-2004,cm-2010,dm-2009,fs-2017,g-2005,hw-2015,l-2010,lrz-2013,lr-2014,ml-2013,os-2011,rxz-2010,rz-2008,rzz-2010,ss-2006,sgds-2010,xz-2009,zz-2017}.
The restriction to compact sets of initial data is due to the fact that their argument is based on weak convergence and if the initial data, $x_n$, do not belong to a compact subset then it is possible that no subsequence of $X^{\e_n,u_n}_{x_n}$ will converge weakly.


Of course, there are many applications, including Freidlin-Wentzell exit problems, where large deviations principles must hold uniformly over non-compact sets. In fact, because bounded subsets of infinite dimensional Banach spaces cannot be compact, Cerrai and R\"ockner's results about uniformity of large deviations principles holding over bounded subsets of initial data \cite{cr-2004} cannot be proved directly via the weak convergence approach.

In \cite{s-2019}, we proved that the variational principle can be used to prove  large deviations principles that are uniform over non-compact and even unbounded subsets of initial data, but we require a stronger notion of convergence of controlled equations than weak convergence. Specifically, if for a set $D \subset C(\bar{\mathcal{O}}\times \{1,...,r\})$ of continuous initial data and for any $\delta>0$ and $N>0$,
\[\lim_{\e \to 0}\sup_{x \in D} \sup_{u \in \mathscr{A}_N} \Pro \left(\left|X^{\e,u}_x-X^{0,u}_x\right|_{C([0,T]\times \bar{\mathcal{O}}\times \{1,...,r\})}>\delta \right) = 0,\]
then $\{X^\e_x\}$ satisfies a uniform large deviations principle that is uniform over $x \in D$. 
In the above expression, $\mathscr{A}_N$ is the set of controls
\[\mathscr{A}_N:= \left\{u \in \mathscr{A}:\Pro \left( \sum_{n=1}^r \int_0^T \int_{\mathcal{O}} |u_n(s,\xi)|^2 d\xi ds \leq N \right)= 1  \right\}.\]
In this paper, we are particularly interested in the case where $D= C(\bar{\mathcal{O}}\times \{1,...,r\})$ is the entire function space.
Using the specific form of the systems of stochastic reaction diffusion equations \eqref{eq:intro-SPDE}, we will be able to prove this kind of convergence in probability holds uniformly over unbounded subsets of initial data, proving our main results.

In Section \ref{S:notation-assumptions} we fix our main notations, present the main assumptions, and define the mild solution. In Section \ref{S:ULDP} we recall the definition of the uniform large deviations principle (ULDP) and we recall the major results about variational representations of infinite dimensional Brownian motion and sufficient conditions that imply uniform large deviations principles. In Section \ref{S:main}, we present the three main results of the paper.  In Section \ref{S:example}, we give an example application of the main results.

Before proving the three main results, in Section \ref{S:M} we study the properties of a fixed-point mapping $\mathcal{M}$ and show that this mapping is well-posed and globally Lipschitz continuous under our weak assumptions that the vector field $f$ is the sum of a decreasing function and a Lipschitz continuous function. In Section \ref{S:exit-uniq}, we prove that the mild solutions to the stochastic reaction-diffusion equations and the controlled stochastic reaction diffusion equation exist and are unique under our weak assumptions. These existence and uniqueness results do not appear elsewhere in the literature.

In Section \ref{S:ULDP-bounded-x}, we prove Theorem \ref{thm:ULDP-bounded-subsets}, which says that Cerrai and R\"ockner's \cite{cr-2004} result about uniformity of the large deviations principle over bounded subsets can be recovered under our weaker assumptions. In Section \ref{S:sigma-bounded}, we prove Theorem \ref{thm:ULDP-sigma-bounded}, which says that when $\sigma$ is uniformly bounded, the large deviations principle is uniform over all initial conditions. In Section \ref{S:super-dissip}, we prove Theorem \ref{thm:ULDP-super-dissip}, which says that when the non-linearity $f$ features super-linear dissipativity and $\sigma$ is unbounded but does not grow too quickly, then the stochastic reaction-diffusion equation satisfies a large deviations principle that is uniform over all initial conditions.

In Appendix \ref{S:appendix-subdiff} we recall results about the left-derivative of a supremum norm for a continuous process. In Appendix \ref{S:stoch-conv} we recall important estimates on the stochastic convolution due to Cerrai \cite{c-2003,c-2009-khasminskii}. In Appendix \ref{S:unif-bounds}, we recall bounds that are uniform with respect to the initial conditions of a stochastic reaction-diffusion equation when the reaction terms features super-linear dissipativity.

\section{Notations and assumptions} \label{S:notation-assumptions}
\subsection{Notations} \label{SS:notations}

For a Euclidean set $A \subset \mathbb{R}^j$ for some $j \in \mathbb{N}$ let $C(A)$ be the set of continuous functions $y:A \to \mathbb{R}$.
Because of the imposed boundary conditions in \eqref{eq:intro-SPDE}, we will work in the spaces of continuous functions with zero boundary conditions. Let
\begin{equation} \label{eq:E-tilde}
  \tilde{E} := \{y \in C(\bar{\mathcal{O}}): y(\xi) = 0 \text{ for } \xi \in \partial \mathcal{O}\}.
\end{equation}
be the space of continuous functions on $\bar{\mathcal{O}}$ with zero boundary conditions endowed with the supremum norm
\begin{equation}
  |y|_{\tilde{E}}: = \sup_{\xi \in \mathcal{O}} |y(\xi)|.
\end{equation}
Any  continuous vector-valued function $x = (x_1, ...., x_r): \bar{\mathcal{O}} \to \mathbb{R}^r$ can be equivalently thought of as a scalar-valued continuous function in the space $C(\bar{\mathcal{O}} \times \{1,...,r\})$. 
Let
\begin{equation} \label{eq:E-def}
  E:= \left\{ x \in C(\bar{\mathcal{O}} \times \{1,...,r\}): x_i (\xi) = 0 \text{ for } i \in \{1,...r\}, \xi \in \partial \mathcal{O}\right\}
\end{equation}
and, for $T>0$,
\begin{equation} \label{eq:E_T-def}
  E_T := \begin{Bmatrix*}[l] \varphi \in C([0,T]\times \bar{ \mathcal{O}} \times \{1,..,r\}):\\ \quad \varphi_i(t,\xi)=0 \text{ for } i \in \{1,...,r\}, t \in [0,T], \xi \in \partial \mathcal{O}  \end{Bmatrix*}.
\end{equation}
$E$ and $E_T$ are endowed with the supremum norms
\begin{equation}
  |x|_E: = \sup_{i \in \{1,...,r\}}\sup_{\xi \in \mathcal{O}} |x_i(\xi)|
\end{equation}
and
\begin{equation}
  |\varphi|_{E_T} : = \sup_{i \in \{1,...,r\}} \sup_{\xi \in \mathcal{O}} \sup_{t \in [0,T]} |\varphi_i(t,\xi)|.
\end{equation}

We remark that this is a slightly different definition than the $E = C (\bar{\mathcal{O}}:\mathbb{R}^r)$ with $|x|_E = \sup_{\xi \in \bar{\mathcal{O}}} \left( \sum_{i=1}^r |x_i(\xi)|^2 \right)^{\frac{1}{2}}$ norm that was used in \cite{cr-2004}. Even though the norms are equivalent, the supremum norm is more convenient for our purposes than the mixture of the supremum and Euclidean norms.

We will show in Theorem \ref{thm:exist-uniq} that the solutions $X^\e$ to the SPDE \eqref{eq:intro-SPDE} exist, are unique, and are $E_T$-valued if their initial data is in $E$.

Throughout the paper we will use other common function spaces including $L^p$ spaces. If no measure is specified, then $L^p$ spaces are defined with respect to the Lebesgue measure on uncountable sets and the counting measure on discrete sets. For example for $p \geq 1$, $L^p([0,T]\times \mathcal{O}\times \{1,..r\})$ is the set of functions $u:[0,T]\times \mathcal{O}\times \{1,...,r\}$ for which the norm
\begin{equation} \label{eq:Lp-norm}
  |u|_{L^p([0,T]\times\mathcal{O}\times\{1,...,r\})} : = \left(\sum_{n=1}^r \int_0^T \int_{\mathcal{O}} |u_n(t,\xi)|^p d\xi dt \right)^{\frac{1}{p}} < +\infty.
\end{equation}
%
%


For any Banach spaces $\mathcal{E}_1, \mathcal{E}_2$, the set $\mathscr{L}(\mathcal{E}_1, \mathcal{E}_2)$ is the space of bounded linear operators from $\mathcal{E}_1$  to $\mathcal{E}_2$. If $\mathcal{E}_1 = \mathcal{E}_2$, then the notation $\mathscr{L}(\mathcal{E}_1) = \mathscr{L}(\mathcal{E}_1, \mathcal{E}_1)$.

\subsection{Main assumptions} \label{SS:assum}
Now we specify our main assumptions about the objects in \eqref{eq:intro-SPDE}. Assumptions \ref{assum:vector-field}, \ref{assum:sigma}, \ref{assum:elliptic-operators}, and \ref{assum:noise} hold throughout the paper. Later in  Section \ref{S:main} we introduce Assumption \ref{assum:sigma-bounded}, which is only used in Theorem \ref{thm:ULDP-sigma-bounded} and Assumption \ref{assum:super-dissip}, which is only used in Theorem \ref{thm:ULDP-super-dissip}.

\begin{assumption}[Vector field] \label{assum:vector-field}
  The vector field $f:[0,+\infty)\times \mathcal{O}\times \mathbb{R}^r \times \{1,...,r\} \to \mathbb{R}$ can be written as
  \begin{equation}
    f_i(t,\xi, u) = g_i(t,\xi, u_i) + h_i(t,\xi, u).
  \end{equation}
  For  any $i \in \{1,...,r\}$, $t\geq 0$, $\xi \in \mathcal{O}$,  the function $\mathbb{R} \ni x \mapsto g_i(t,\xi,x)$ is continuous and decreasing in the sense that for any $x,y \in \mathbb{R}$ such that $x> y $
  \begin{equation} \label{eq:g-decr}
     g_i(t,\xi,x) - g_i(t,\xi,y) \leq  0,
  \end{equation}
  There exists an increasing function $L:[0,+\infty) \to [0,+\infty)$ such that for any $x, y \in \mathbb{R}^r$,
  \begin{equation} \label{eq:h-Lip-assum}
    \sup_{i \in \{1,...,r\}}\sup_{s \in [0,t]} \sup_{\xi \in \mathcal{O}} |h_i(s,\xi,x) - h_i(s,\xi,y)| \leq L(t) \sup_{i \in \{1,...,r\}}|x_i - y_i|
  \end{equation}
  and
  \begin{equation} \label{eq:h-lin-growth}
    \sup_{i \in \{1,...,r\}}\sup_{s \in [0,t]} \sup_{\xi \in \mathcal{O}} |h_i(s,\xi,x) | \leq L(t)\left( 1+  \sup_{i \in \{1,...,r\}}|x_i| \right).
  \end{equation}
\end{assumption}

%

\begin{assumption}[Multiplicative noise coefficient] \label{assum:sigma}
  There exists an increasing function $L:[0,+\infty) \to [0,+\infty)$ such that for all $x,y \in \mathbb{R}^r$,
  \begin{equation} \label{eq:sigma-Lip}
    \sup_{i,n \in \{1,...,r\}} \sup_{s \in [0,t]} \sup_{\xi \in \mathcal{O}} |\sigma_{in}(s,\xi,x) - \sigma_{in}(s,\xi,y)| \leq L(t) \sup_{i \in \{1,...,r\}}|x_i - y_i|
  \end{equation}
  and for any $x \in \mathbb{R}^r$
  \begin{equation} \label{eq:sigma-lin-growth}
    \sup_{i,n \in \{1,...,r\}} \sup_{s \in [0,t]} \sup_{\xi \in \mathcal{O}} |\sigma_{in}(s,\xi,x) | \leq L(t) \left(1 +  \sup_{i \in \{1,...,r\}}|x_i|\right)
  \end{equation}
\end{assumption}

\begin{assumption}[Elliptic operators] \label{assum:elliptic-operators}
  The spatial domain $\mathcal{O}\subset \mathbb{R}^d$ is open, bounded, and has smooth boundary.
  For $i \in \{1,...,r\}$, the second-order elliptic operators $\mathcal{A}_i$ are of the form
  \begin{equation} \label{eq:elliptic-ops}
    \mathcal{A}_i \varphi(\xi):= \sum_{j=1}^d \sum_{k=1}^d a^i_{jk}(\xi) \frac{\partial^2 \varphi}{\partial \xi_j \partial \xi_k}(\xi) + \sum_{j=1}^d b^i_j(\xi) \frac{\partial \varphi}{\partial \xi_j}(\xi).
  \end{equation}
  In the above expression, $a^i_{jk}:\bar{\mathcal{O}} \to \mathbb{R}$ are continuously differentiable and $b^i_j:\bar{\mathcal{O}} \to \mathbb{R}$ are continuous.
  The matrix $(a^i_{jk}(\xi))_{jk}$ is symmetric and uniformly elliptic in the sense that there exists $\kappa>0$ such that for any vector $(x_1,...,x_d)$, 
  \begin{equation} \label{eq:elliptic}
    \inf_{\xi \in \mathcal{O}} \inf_{i \in \{1,...,r\}}\sum_{j=1}^d\sum_{k=1}^d a^i_{jk}(\xi)x_j x_k \geq \kappa \sum_{j=1}^d x_j^2.
  \end{equation}
\end{assumption}

\begin{proposition} \label{prop:A_i}
  The operators $\mathcal{A}_i$ can be written as
  \[\mathcal{A}_i = \mathcal{B}_i + \mathcal{L}_i\]
  where
  \begin{equation} \label{eq:symmetric-operator}
    \mathcal{B}_i \varphi(\xi):= \sum_{j=1}^d \sum_{k=1}^d \frac{\partial}{\partial \xi_k} \left(a^i_{jk}(\xi) \frac{\partial \varphi}{\partial \xi_j}(\xi) \right)
  \end{equation}
  is self-adjoint and
  \begin{equation}
    \mathcal{L}_i \varphi(\xi) := \sum_{j=1}^d \left( b^i_j(\xi) - \sum_{k=1}^d \frac{\partial a^i_{jk}}{\partial \xi_k}(\xi) \right) \frac{\partial \varphi}{\partial \xi_j}(\xi)
  \end{equation}
  is a first-order differential operator.

  For each $i \in \{1,...,r\}$, there exists an orthonormal system of eigenvectors $\{e_{i,k}\}_{k=1}^\infty \subset L^2(\mathcal{O})$ and eigenvalues $\{\alpha_{i,k}\}_{k=1}^\infty$ such that the realization $B_i$ of $\mathcal{B}_i$ in $L^2(\mathcal{O})$ with the imposed boundary conditions satisfies
  \begin{equation} \label{eq:eigen}
    B_i e_{i,k} = -\alpha_{i,k} e_{i,k}.
  \end{equation}
  The eigenvalues are non-negative, diverge to infinity, and can be written in increasing order $0\leq  \alpha_{i,k} \leq \alpha_{i,k+1}$. By elliptic regularity results, for fixed $k,i$, $e_{i,k} \in \tilde E$ defined in \eqref{eq:E-tilde}. See, for example,  \cite[Chapter 6.5]{evans}.
\end{proposition}

\begin{assumption}[Noise] \label{assum:noise}
  Fix a filtered probability space $(\Omega, \mathcal{F}, \{\mathcal{F}_t\}, \Pro)$.
  The driving noise $w=(w_1,...,w_r)$ can be formally written as the sum
  \begin{equation} \label{eq:noise}
    w_n(t,\xi):= \sum_{j=1}^\infty \lambda_{n,j} W_{n,j}(t) f_{n,j}(\xi)
  \end{equation}
  where for fixed $n \in \{1,...,r\}$, $\{f_{n,j}\}_{j=1}^\infty$ is an orthonormal basis of $L^2(\mathcal{O})$, and for each fixed $n,j$ $f_{n,j} \in \tilde E$. Such a sequence $f_{n,j}$ exists because one could take $f_{n,j}:=e_{n,j}$ \eqref{eq:eigen}.  $\{\{W_{n,j}\}_{j=1}^\infty\}_{n=1}^{r}$ is a countable collection of independent one-dimensional Brownian motions on  $(\Omega,\mathcal{F}, \{\mathcal{F}_t\},\Pro)$. The numbers $\lambda_{n,j}\geq 0$ and there exists $\beta \in (0,1)$ and $\rho \in [2,+\infty]$ such that
  \begin{equation}\label{eq:beta-rho-relation}
    \frac{\beta(\rho - 2)}{\rho} <1,
  \end{equation}
  \begin{equation} \label{eq:alpha-sum}
    \sum_{i=1}^r\sum_{k=1}^\infty \alpha_{i,k}^{-\beta} |e_{i,k}|_{\tilde{E}}^2 <\infty,
  \end{equation}
  and
  \begin{align} \label{eq:lambda-sum}
    &\sum_{j=1}^\infty \sum_{n=1}^r \lambda_{n,j}^\rho |f_{n,j}|_{\tilde{E}}^2 <\infty, \text{ if } \rho<+\infty \\ &\text{ or }\sup_j \sup_n \lambda_{n,j} <+\infty,  \text{ if } \rho=+\infty ,
  \end{align}
  where $\alpha_{i,k}$, $e_{i,k}$ are the eigenvalues of $B_i$ from \eqref{eq:eigen}.
\end{assumption}
\begin{remark}
  Cerrai \cite{c-2003,c-2009-khasminskii} proved that Assumption \ref{assum:noise} is a sufficient condition that implies that mild solutions to the stochastic reaction-diffusion equation are continuous functions of space and time.
\end{remark}
  For $n \in \{1,..r\}$, let $Q_n: L^2(\mathcal{O}) \to L^2(\mathcal{O})$ be the bounded linear operator
  \begin{equation} \label{eq:Q_n-def}
    Q_n f_{n,j} = \lambda_{n,j} f_{n,j}
  \end{equation}
  and let $Q: L^2(\mathcal{O}\times \{1,...r\}) \to L^2(\mathcal{O}\times \{1,...r\})$ be defined so that for any $n \in \{1,...,r\}$, and $\xi \in \mathcal{O}$,
  \begin{equation} \label{eq:Q-def}
    [Qu]_n(\xi) = [Q_n u_n](\xi).
  \end{equation}

\subsection{Semigroups and mild solution} \label{SS:semigroup}
Let $S_i(t)$ be the semigroup on $\tilde{E}$ \eqref{eq:E-tilde} generated by the elliptic operator $\mathcal{A}_i$ with zero boundary conditions. It is standard that $S_i(t)$ is a $C_0$ semigroup on $\tilde{E}$ (see \cite{evans}).

For $x \in E$, let $[S(t)x]_i(\xi): = [S_i(t)x_i](\xi)$. In this way, $S(t): E \to E$ is a $C_0$ contraction semigroup on $E$.

The mild solution for $X^\e_{x,i}$ is defined to be the solution to the system of integral equations for $i \in \{1,...r\}$,
\begin{align} \label{eq:mild-def-component}
  X^\e_{x,i}(t) = &S_i(t)x_i + \int_0^t S_i(t-s)F_i(s,X^\e_x(s))ds \nonumber\\
  &+ \sqrt{\e} \sum_{n=1}^r \int_0^t S_i(t-s) R_{in}(s,X^\e_x(s))dw_n(s).
\end{align}
In the above equation, the spatial variable $\xi$ has been suppressed. \\$F_i: [0,+\infty)\times E \to \tilde{E}$ is the Nemytskii operator where for any $i \in \{1,...r\}$, $t>0$, $\xi \in \mathcal{O}$, and $x \in E$,
\begin{equation} \label{eq:Nemytskii-F}
  [F_i(t,x)](\xi) : = f_i(t,\xi,x(\xi)),
\end{equation}
and
$R_{in}: [0,+\infty)\times E \to \mathscr{L}(L^2(\mathcal{O}))$ is defined such that for any $i,n \in \{1,..r\}$, $t>0$, $\xi \in \mathcal{O}$, and $x \in E$, and $h \in L^2(\mathcal{O})$,
\begin{equation} \label{eq:Nemytskii-R}
  [R_{in}(t,x)h](\xi) = \sigma_{in}(t,\xi,x(\xi))h(\xi).
\end{equation}

By the definition of the noise \eqref{eq:noise}, the stochastic convolution can be understood as the infinite sum of one-dimensional Ito integrals
\[\int_0^t S_i(t-s) R_{in}(s,X^\e_x(s))dw_n(s) = \sum_{j=1}^\infty  \int_0^t S_i(t-s) R_{in}(s,X^{\e}_x(s))   \lambda_{n,j} f_{n,j} dW_{n,j}(s). \]
The properties of the stochastic convolution can be found in \cite{c-2003} and are included in Appendix \ref{S:stoch-conv} below.

\begin{definition} \label{def:mild}
The \textit{mild solution} to \eqref{eq:intro-SPDE} is defined to be the $E_T$-valued solution to
\begin{equation} \label{eq:mild-def}
  X^\e_{x}(t) = S(t)x + \int_0^t S(t-s)F(s,X^\e_x(s))ds + \sqrt{\e} \int_0^t S(t-s) R(s,X^\e_x(s))dw(s)
\end{equation}
In the above equation $F: [0,+\infty)\times E \to E$ is the vector $F=(F_1,...,F_r)$ and $R:[0,+\infty) \times E \to \mathscr{L}(L^2(\mathcal{O}\times \{1,...,r\}))$ is the matrix $R = (R_{in})_{in}$. $w = (w_1,...w_r)$. We prove that there exists a unique mild solution in Section \ref{S:exit-uniq}.
\end{definition}

To prove the large deviations results we will study the convergence properties of mild solutions to the stochastic control problems \eqref{eq:intro-contolled-SPDE}. The mild solution to \eqref{eq:intro-contolled-SPDE} will solve the integral equation
\begin{align} \label{eq:controlled-mild}
  X^{\e,u}_x(t) = &S(t)x + \int_0^t S(t-s)F(s,X^{\e,u}_x(s))ds \nonumber\\
  &+ \sqrt{\e}\int_0^t S(t-s)R(s,X^{\e,u}_x(s))dw(s)\nonumber\\
  &+ \int_0^t S(t-s)R(s,X^{\e,u}_x(s))Qu(s)ds.
\end{align}

\section{Uniform large deviations principle and the equicontinuous uniform Laplace principle} \label{S:ULDP}
In this section we recall the definitions of Freidlin and Wentzell's uniform large deviations principle (ULDP) and a result from \cite{s-2019} that proves that the uniform convergence in probability of certain controlled process implies that a collection of processes satisfies the ULDP.

Let $(\mathcal{E},\vartheta)$ be a Polish space and let $\mathcal{E}_0$ be a set used for indexing (there are no topological assumptions on $\mathcal{E}_0$). When we apply these results in the sequel, we will set $\mathcal{E} =E_T$ and $\mathcal{E}_0 =E$. Let $\{Y^\e_x\}_{x \in \mathcal{E}_0, \e>0}$ be a collection of $\mathcal{E}$-valued random variables. For every $x \in \mathcal{E}_0$, let $I_x: \mathcal{E} \to [0,+\infty]$ be a lower-semicontinuous function called a rate function. For each $x \in \mathcal{E}_0$ and $s \geq 0$, let
\[\Phi_x(s): = \left\{\varphi \in \mathcal{E}: I_x(\varphi)\leq s \right\}\]
be the level sets of the rate function. 
Let $\dist_{\mathcal{E}}: \mathcal{E} \times 2^{\mathcal{E}} \to [0,+\infty)$ be defined as the minimal distance between an element of $\mathcal{E}$ and a set
\begin{equation} \label{eq:dist-def}
  \textnormal{dist}_\mathcal{E}(\varphi, \Psi) := \inf_{\psi \in \Psi} \vartheta(\varphi,\psi).
\end{equation}

\begin{definition}[Uniform large deviations principle (ULDP) (Section 3.3 of \cite{fw})] \label{def:ULDP}
  A family $\{Y^\e_x\}_{x \in \mathcal{E}_0,\e>0}$ of $\mathcal{E}$-valued random variables satisfies a uniform large deviations principle uniformly over a set $D \subset \mathcal{E}_0$ with respect to the rate functions $I_x$ if
  \begin{enumerate}
    \item for any $\delta>0$ and  $s_0 \geq 0$, 
      \begin{equation} \label{eq:ULDP-lower}
        \liminf_{\e \to 0} \inf_{x \in D} \inf_{\varphi \in \Phi_x(s_0)} \left(\e \log \Pro\left( \vartheta(Y^\e_x, \varphi)<\delta \right) + I_x(\varphi) \right) \geq 0.
      \end{equation}
    \noindent and
    \item for any $\delta>0$ and  $s_0 \geq 0$, 
      \begin{equation} \label{eq:ULDP-upper}
        \limsup_{\e \to 0} \sup_{x \in D} \sup_{s \in [0,s_0]} \left(\e \log \Pro\left(\dist_{\mathcal{E}}\left(Y^\e_x, \Phi_x(s) \right) \geq \delta \right) + s \right) \leq 0.
      \end{equation}
  \end{enumerate}
\end{definition}

%
%

The following theorem describes a sufficient condition that implies that measurable functions of a countable collection of Brownian motions satisfy the ULDP. 

Suppose that $W = \{W_j(\cdot)\}_{j=1}^\infty$ is a countable collection of i.i.d. one-dimensional Brownian motions on a filtered probability space $(\Omega, \mathcal{F}, \{\mathcal{F}_t\}, \Pro)$. For fixed $T>0$ and any $x \in \mathcal{E}_0$ suppose that $\mathscr{G}_x: C([0,T]\times \mathbb{N}) \to \mathcal{E}$ is a measurable mapping. For $\e\geq 0$ and $x \in \mathcal{E}_0$, let
 \begin{equation}
   Y^\e_x := \mathcal{G}_x(\sqrt{\e}W)
 \end{equation}
 For each $N>0$, let $\mathscr{B}_N$ be the collection of $u \in L^2(\Omega\times [0,T] \times \mathbb{N})$ that are adapted to the filtration $\mathcal{F}_t$ and satisfy
 \begin{equation} \label{eq:B_N-def}
   \Pro \left(\sum_{j=1}^\infty\int_0^T |u_j(s)|^2ds \leq N  \right)=1.
 \end{equation}

 For each $u \in \mathscr{B}_N$, let $Y^{\e,u}_x$ denote the controlled $\mathcal{E}$-valued random variable
 \begin{equation}
   Y^{\e,u}_x := \mathcal{G}_x \left(\sqrt{\e} W + \int_0^\cdot u(s)ds \right).
 \end{equation}
 
 Proving that a family $Y^\e_x$ satisfies a ULDP directly using Definition \ref{def:ULDP} can be cumbersome. The proofs of the main results in this paper are based on Theorem 2.13 of \cite{s-2019}, which proves that uniform convergence in probability of the controlled system $Y^{\e,u}_x$ to $Y^{0,u}_x$ as $\e \to 0$ implies the ULDP.
 
\begin{theorem}[Theorem 2.13 of \cite{s-2019}] \label{thm:ULDP-suff-cond-abstract}
  Let $D \subset \mathcal{E}_0$. If for any $\delta>0$ and $N>0$, 
  \begin{equation}
    \lim_{\e \to 0} \sup_{x \in D} \sup_{u \in \mathscr{B}_N} \Pro \left( \vartheta \left(Y^{\e,u}_x,Y^{0,u}_x \right)>\delta \right)=0,
  \end{equation}
  then the family $\{Y^\e_x\}$ satisfies a ULDP uniformly over $D$ with respect to the rate functions $I_x: \mathcal{E} \to [0,+\infty]$ defined by
  \begin{equation}
    I_x(\varphi):= \inf \left\{\frac{1}{2}\sum_{j=1}^\infty \int_0^T |u_j(s)|^2ds: \varphi = Y^{0,u}_x, u \in L^2([0,T]\times \mathbb{N})  \right\}.
  \end{equation}
\end{theorem}

In the context of the system of reaction-diffusion equations \eqref{eq:mild-def}, we will let $\mathcal{E}_0 := E $ defined in \eqref{eq:E-def} be the set of initial data and for a fixed time horizon $T>0$ let $\mathcal{E}=E_T$ defined in \eqref{eq:E_T-def}. Because the driving noise is defined in terms of a countable collection of i.i.d. Brownian motions $W =\{\{W_{n,j}\}_{j=1}^\infty\}_{n=1}^r$ (see \eqref{eq:noise}) and because we will show that the mild solutions \eqref{eq:mild-def} exist and are unique (see Corollary \ref{cor:exist-uniq}), there exists a measurable mapping $\mathscr{G}_x: C([0,T]\times \{1,...,r\}\times \mathbb{N}) \to \mathcal{E}$ such that $X^\e_x := \mathscr{G}_x(\sqrt{\e}W)$ solves \eqref{eq:mild-def}.

According to \eqref{eq:B_N-def}, the spaces $\mathscr{B}_N$ will consist of adapted processes  $u \in L^2(\Omega \times [0,T] \times \{1,...,r\} \times \mathbb{N})$ satisfying
\begin{equation}
  \Pro \left(\sum_{j=1}^\infty \sum_{n=1}^r \int_0^T |u_{n,j}(s)|^2ds \leq N \right) =1.
\end{equation}

For $N>0$ and $u \in \mathscr{B}_N$, the controlled processes $Y^{\e,u}_x = \mathscr{G}_x\left(\sqrt{\e}W + \int_0^\cdot u(s)ds \right)$ solves the integral equation
\begin{align} \label{eq:controlled-mild-w-isom}
  Y^{\e,u}_x(t) = &S(t)x + \int_0^t S(t-s)F(s,Y^{\e,u}_x(s))ds \nonumber\\
  &+ \sqrt{\e}\int_0^t S(t-s)R(s,Y^{\e,u}_x(s))dw(s)\nonumber\\
  &+ \int_0^t S(t-s)R(s,Y^{\e,u}_x(s))Q\mathcal{I}u(s)ds
\end{align}
where $\mathcal{I}: L^2( \{1,...,r\}\times \mathbb{N}) \to L^2(\mathcal{O} \times \{1,...,r\})$ is the isometry defined by
\[[\mathcal{I}u]_n(\xi): = \sum_{j=1}^\infty u_{n,j} f_{n,j}(\xi).\]
In the above expression, $f_{n,j}$ are the orthonormal basis defined in Assumption \ref{assum:noise}. The noise $w(t)$ is defined in Assumption \ref{assum:noise}.
$F$ and $R$ are the Nemytskii operators defined in \eqref{eq:Nemytskii-F} and \eqref{eq:Nemytskii-R}.

Because $\mathcal{I}$ is an isometry, we can equivalently define $\mathscr{A}_N: = \mathcal{I}(\mathscr{B}_N)$ to be the family of adapted $L^2([0,T]\times \mathcal{O}\times \{1,...r\})$ processes satisfying
\begin{equation} \label{eq:A_N-def}
  \Pro \left(\sum_{n=1}^r \int_0^T \int_{\mathcal{O}} |u_n(s,\xi)|^2d\xi ds \leq N\right) = 1.
\end{equation}
and then define $X^{\e,u}_x$ for $u \in \mathscr{A}_N$,
\begin{align} 
  X^{\e,u}_x(t) = &S(t)x + \int_0^t S(t-s)F(s,X^{\e,u}_x(s))ds \nonumber\\
  &+ \sqrt{\e}\int_0^t S(t-s)R(s,X^{\e,u}_x(s))dw(s)\nonumber\\
  &+ \int_0^t S(t-s)R(s,X^{\e,u}_x(s))Qu(s)ds.
\end{align}
this agrees with \eqref{eq:controlled-mild}.

The isometry between $\mathscr{B}_N$ and $\mathscr{A_N}$ and Theorem \ref{thm:ULDP-suff-cond-abstract} imply the following result that we will use to prove our three main results.
\begin{corollary} \label{cor:ULDP-suff-cond}
  Let $D$ be a  subset of $E$. If for $T>0$ and any $\delta>0$, and $N>0$,
  \begin{equation}
    \lim_{\e \to 0} \sup_{x \in D} \sup_{u \in \mathscr{B}_N} \Pro \left(  \left|X^{\e,u}_x-X^{0,u}_x \right|_{E_T}>\delta \right)=0,
  \end{equation}
  then the family $\{X^\e_x\}$ satisfies a ULDP in the $E_T$ norm uniformly over $D$ with respect to the rate functions $I_{x,T}: \mathcal{E} \to [0,+\infty]$ defined by
  \begin{equation} \label{eq:rate-fct}
    I_{x,T}(\varphi):= \inf \left\{\frac{1}{2}\sum_{n=1}^r \int_0^T \int_{\mathcal{O}}|u_n(s,\xi)|^2d \xi ds: \varphi = X^{0,u}_x\right\}
  \end{equation}
  where the infimum is taken over all $u \in L^2([0,T]\times \mathcal{O}\times \{1,...,r\})$.
\end{corollary}

\section{Main results} \label{S:main}

The first main result of this paper proves that under Assumptions \ref{assum:vector-field}, \ref{assum:sigma}, \ref{assum:elliptic-operators}, and \ref{assum:noise}, the mild solutions $X^\e_x$ satisfy a large deviations principle that is uniform over bounded subsets of initial data $x$. This result generalizes the result of Cerrai and R\"ockner \cite{cr-2004} by removing the restrictions to locally Lipschitz continuity and polynomial growth of the reaction term $f$. For these results, recall Definition \ref{def:ULDP} of the ULDP and the definitions of the rate function $I_{x,T}$ \eqref{eq:rate-fct} and define the level sets for $s\geq 0$,
\begin{equation} \label{eq:level-sets}
  \Phi_{x,T}(s) := \left\{\varphi \in E_T : I_{x,T}(\varphi)\leq s \right\}.
\end{equation}

\begin{theorem} \label{thm:ULDP-bounded-subsets}
  Assume Assumptions \ref{assum:vector-field}, \ref{assum:sigma}, \ref{assum:elliptic-operators}, and \ref{assum:noise}. For any fixed $T>0$, $X^\e_x$ satisfy a large deviations principle in $E_T$ that is uniform over bounded subsets of initial data. In particular, for any $K>0$, any $\delta>0$, and any $s_0\geq 0$,
  \begin{equation} \label{eq:ULDP-lower-bounded-x}
    \liminf_{\e \to 0} \inf_{|x|_E \leq K} \inf_{\varphi \in \Phi_{x,T}(s_0)} \left( \e \log \Pro \left(\left| X^\e_x - \varphi \right|_{E_T} <\delta \right) + I_{x,T}(\varphi) \right) \geq 0
  \end{equation}
  and
  \begin{equation}
    \limsup_{\e \to 0} \sup_{|x|_E \leq K} \sup_{s \in [0,s_0]} \left(\e \log \Pro\left(\dist_{E_T}\left(X^\e_x, \Phi_{x,T}(s) \right) \geq \delta \right) + s \right) \leq 0.
  \end{equation}
\end{theorem}

The proof of Theorem \ref{thm:ULDP-bounded-subsets} is in Section \ref{S:ULDP-bounded-x}.

The next result shows that if we restrict the multiplicative noise coefficients $\sigma_{in}$ to be uniformly bounded, then the large deviations principle actually holds uniformly over \textit{unbounded} subsets of initial data.
We continue to assume Assumptions \ref{assum:vector-field}, \ref{assum:elliptic-operators}, and \ref{assum:noise} and we add the following strengthening of Assumption \ref{assum:sigma}.

\begin{assumption}[Bounded $\sigma$] \label{assum:sigma-bounded}
  There exists an increasing function $L:[0,+\infty) \to [0,+\infty)$ such that for all $x,y  \in \mathbb{R}^r$,
  \begin{equation} \label{eq:sigma-Lip-bounded}
    \sup_{i,j \in \{1,...,r\}} \sup_{s \in [0,t]} \sup_{\xi \in \mathcal{O}} |\sigma_{ij}(s,\xi,x) - \sigma_{ij}(s,\xi,y)| \leq L(t) \sup_{i \in \{1,...,r\}}|x_i-y_i|
  \end{equation}
  and
  \begin{equation} \label{eq:sigma-bounded}
    \sup_{i,j \in \{1,...,r\}} \sup_{s \in [0,t]} \sup_{\xi \in \mathcal{O}} \sup_{x \in \mathbb{R}^r} |\sigma_{ij}(s,\xi,x) | \leq L(t).
  \end{equation}
\end{assumption}

\begin{theorem} \label{thm:ULDP-sigma-bounded}
  Assume Assumptions \ref{assum:vector-field}, \ref{assum:elliptic-operators}, \ref{assum:noise}, and \ref{assum:sigma-bounded}.
  For any fixed $T>0$, $X^\e_x$ satisfy a large deviations principle in $E_T$ that is uniform over all initial conditions in $E$. In particular, for any $\delta>0, s_0 \geq 0$,
  \begin{equation} 
    \liminf_{\e \to 0}\inf_{x \in E} \inf_{\varphi \in \Phi_x(s_0)} \left(  \e \log \Pro \left(|X^\e_x - \varphi|_{E_T}<\delta \right)  + I_x(\varphi)  \right) \geq 0,
  \end{equation}
  and
  \begin{equation} 
    \limsup_{\e \to 0} \sup_{x \in E} \sup_{s \in [0,s_0]} \left(\e \log \Pro \left(\dist_{E_T}(X^\e_x, \Phi_x(s)) \geq \delta \right) + s \right) \leq 0.
  \end{equation}
\end{theorem}

The proof of Theorem \ref{thm:ULDP-sigma-bounded} is in Section \ref{S:sigma-bounded}.

The third main result identifies a sufficient condition that implies that the large deviations principle holds uniformly over all initial data even when $\sigma$ is unbounded. The result requires the reaction term $f$ to feature sufficiently strong superlinearly dissipativity to counteract the expansive effects of the unbounded $\sigma$. Specifically we assume the following.

\begin{assumption}[Super-linear dissipativity] \label{assum:super-dissip}
  The reaction term $f$ can be written as $f_i=g_i + h_i$ where $g_i$ and $h_i$ satisfy Assumption \ref{assum:vector-field}. Additionally, there exists $m>1$ (not necessarily an integer), $\mu>0$, and $c_0>0$ such that for any $i \in \{1,...,r\}$, $t>0$,  $\xi \in \bar{\mathcal{O}}$, and $|v_i| >c_0$,
  \begin{equation} \label{eq:g-super-dissip}
    g_i(t,\xi,v_i) \sgn(v_i) \leq - \mu |v_i|^m.
  \end{equation}

  We further assume  that there exists an increasing function $L:[0,+\infty) \to [0,+\infty)$ such that for any $x, y\in \mathbb{R}^r$,
  \begin{equation} \label{eq:sigma-Lip-super-dissip}
    \sup_{i,j \in \{1,...,r\}} \sup_{s \in [0,t]} \sup_{\xi \in \mathcal{O}} |\sigma_{ij}(s,\xi,x) - \sigma_{ij}(s,\xi,y)| \leq L(t) \sup_{i \in \{1,...,r\}} |x_i-y_i|.
  \end{equation}
  and that there exists
  \begin{equation} \label{eq:nu}
    \nu \in \left[ 0, \frac{m-1}{2} \left(1 - \frac{\beta(\rho-2)}{\rho} \right) \right) \cap [0,1]
  \end{equation}
  such that for any $x \in \mathbb{R}^r$,
  \begin{equation} \label{eq:sigma-growth-super-dissip}
  \sup_{i,n \in \{1,...,r\}}\sup_{\xi \in \mathcal{O}}\sup_{t \in [0,T]}|\sigma_{in}(t,\xi,x)| \leq L(T) \left(1 + \sup_{i \in \{1,...,r\}}|x_i|\right)^{\nu}.  \end{equation}
\end{assumption}

%
%
%

\begin{theorem} \label{thm:ULDP-super-dissip}
  Assume Assumptions \ref{assum:vector-field},  \ref{assum:elliptic-operators}, \ref{assum:noise}, and \ref{assum:super-dissip}. For any fixed $T>0$
  $X^\e_x$ satisfy a large deviations principle in $E_T$ that is uniform over all initial conditions in $E$. In particular, for any $\delta>0, s_0 \geq 0$,
  \begin{equation} \label{eq:LDP-lower}
    \liminf_{\e \to 0}\inf_{x \in E} \inf_{\varphi \in \Phi_x(s_0)} \left(  \e \log \Pro \left(|X^\e_x - \varphi|_{E_T}<\delta \right)  + I_{x,T}(\varphi)  \right) \geq 0,
  \end{equation}
  \begin{equation} \label{eq:LDP-upper}
    \limsup_{\e \to 0} \sup_{x \in E} \sup_{s \in [0,s_0]} \left(\e \log \Pro \left(\dist_{E_T}(X^\e_x, \Phi_{x,T}(s)) \geq \delta \right) + s \right) \leq 0.
  \end{equation}
\end{theorem}

The proof of Theorem \ref{thm:ULDP-super-dissip} is in Section \ref{S:super-dissip}.

\section{Example: System of stochastic reaction-diffusion equations exposed to space-time white noise in spatial dimension $d=1$} \label{S:example}
We consider a class of reaction-diffusion equations with polynomially dissipative forcing and polynomially growing multiplicative noise term in spatial dimension $d=1$. For simplicity, we only consider one equation ($r=1$), rather than a system of equations. Let $m\geq0$ and $\nu \leq 1$.  $m$ does not need to be an integer. Let $X^\e_x(t,\xi)$ be the  mild solution to
\begin{equation}
  \begin{cases}
    \frac{\partial}{\partial t} X^\e_{x}(t,\xi)  = \frac{\partial^2}{\partial \xi^2} X^\e_{x}(t,\xi) -|X^\e_x(t,\xi)|^m\sgn(X^\e_x(t,\xi))   \\
    \hspace{2.5cm}+\sqrt{\e} \left(1+|X^\e_x(t,\xi)|\right)^\nu\frac{\partial w}{\partial t}(t,\xi)\\
    X^\e_x(0,\xi) = x(\xi)\\
    X^\e_x(t,0) = X^\e_x(t,\pi) = 0
  \end{cases}
\end{equation}
defined on the one-dimensional spatial domain $\mathcal{O} = (0, \pi)$. $\frac{\partial w}{\partial t}$ is a space-time white noise.
%
%
%
%
%

In this spatial dimension $d=1$ setting, the eigenvalues of $\frac{\partial^2}{\partial \xi^2}$ are $-\alpha_k$ where $\alpha_{k} = k^2$. Because $\frac{\partial w}{\partial t}$ is a space-time white noise, $\lambda_{j}\equiv 1$. These sequences satisfy Assumption \ref{assum:noise} with $\rho=+\infty$ and any $\beta \in \left( \frac{1}{2}, 1 \right)$.

For any $m \geq 0$, the function $g(x) = -|x|^m \sgn(x)$ is decreasing. For any $\nu \leq 1$, $\sigma(x):=(1+|x|)^\nu$ is Lipschitz continuous. Therefore, if $m \geq 0$ and $\nu\leq 1$, Theorem \ref{thm:ULDP-bounded-subsets} guarantees that $X^\e_x$ satisfies a ULDP that is uniform over bounded subsets of initial data.

If $\nu\leq 0$ (the case of bounded noise coefficients) and $m\geq 0$, then Theorem \ref{thm:ULDP-sigma-bounded} guarantees that the system satisfies a uniform large deviations principle that is uniform over all $E$-valued initial data.

When $\nu$ satisfies
\begin{equation}
  \nu < \frac{(m-1)(1 - \beta)}{2} < \frac{m-1}{4} \ \ \ \ \  \text{ and } \ \ \ \ \  \nu \leq 1 ,
\end{equation}
Theorem \ref{thm:ULDP-super-dissip} guarantees that $\{X^\e_x\}$ will satisfy a large deviations principle that is uniform over all $E$-valued initial data.
The restriction  $\nu\leq 1$ is required because $\sigma(x) = (1 + |x|)^\nu$ fails to be globally Lipschitz continuous if $\nu>1$.

If $m=3$ and $\nu< \frac{1}{2}$,  then Theorem \ref{thm:ULDP-super-dissip} guarantees that $X^\e_x$ will satisfy a uniform large deviations principle that is uniform over all continuous initial data. If $m=5$ and $\nu<1$, then the large deviations principle will hold uniformly over all $E$-valued data.
If $m>5$ and $\nu\leq 1$ then $X^\e_x$ will satisfy a uniform large deviations principle that is uniform over all $E$-valued initial conditions.



\section{Lipschitz continuity of $\mathcal{M}$} \label{S:M}

In order to prove Theorems \ref{thm:ULDP-bounded-subsets}, \ref{thm:ULDP-sigma-bounded}, and \ref{thm:ULDP-super-dissip}, and even to prove that the mild solutions to \eqref{eq:mild-def} and \eqref{eq:controlled-mild} are well defined, we define a mapping $\mathcal{M}: E_T \to E_T$ that sends an element $z \in E_T$ to the fixed point solution
\begin{equation} \label{eq:M-def}
 \mathcal{M}(z)(t): = \int_0^t S(t-s) F(s,\mathcal{M}(z)(s))ds + z(t).
\end{equation}

In this section we prove that $\mathcal{M}$ is well-defined and globally Lipschitz continuous whenever $f$ satisfies Assumption \ref{assum:vector-field}, even if $f$  fails to be locally Lipschitz continuous.

The mapping $\mathcal{M}$ is essential to our investigation of the mild solutions to the reaction-diffusion equations because $X^\e_x$ will be a mild solution solving \eqref{eq:mild-def} if and only if
\[X^\e_x = \mathcal{M}(S(\cdot)x + \sqrt{\e}Z^\e_x)\]
where
\[Z^\e_x(t) = \int_0^t S(t-s)R(s,X^\e_x(s))dw(s).\]
Similarly, $X^{\e,u}_x$ solves \eqref{eq:controlled-mild} if and only if
\[X^{\e,u}_x = \mathcal{M}(S(\cdot)x + Y^{\e,u}_x + \sqrt{\e}Z^{\e,u}_x)\]
where
\[Y^{\e,u}_x(t) = \int_0^t S(t-s) R(s,X^{\e,u}_x(s))Qu(s)ds\]
and
\[Z^{\e,u}_x(t) = \int_0^t S(t-s)R(s,X^{\e,u}_x(s))dw(s).\]

\begin{theorem}
  For any $z \in E_T$, there exists a solution $\mathcal{M}(z) \in E_T$ to \eqref{eq:M-def}.
\end{theorem}

\begin{proof}
  Let $g_i$ be the non-increasing functions from Assumption \ref{assum:vector-field}. For $n \in \mathbb{N}$, $t\geq 0$ and $\xi \in \mathcal{O}$ define $x \mapsto g_{i,n}(t,\xi,x)$ to be the Yosida approximation
  \[g_{i,n}(t,\xi,x) := n\left(J_{i,n}(t,\xi,x)  -x \right), \ \ \  J_{i,n}(t,\xi,x) = \left(\bullet - \frac{1}{n} g_{i}(t,\xi,\bullet)\right)^{-1}(x). \]
  Let $f_{i,n}(t,\xi,x) := g_{i,n}(t,\xi,x_i) + h_i(t,\xi,x)$.

  According to \cite[Proposition D.11]{dpz}, for $t\geq 0$, $ \xi \in \bar{\mathcal{O}}$, $i \in \{1,...,r\}$, and $x,y \in \mathbb{R}$
  \begin{align}
    &|g_{i,n}(t,\xi,x) - g_{i,n}(t,\xi,y)| \leq 2n |x-y|\\
    &|g_{i,n}(t,\xi,x)| \leq |g_{i}(t,\xi,x)|\\
    &g_{i,n}(t,\xi,x) - g_{i,n}(t,\xi,y) \leq 0 \text{ for } x>y,\\
    &\lim_{n \to +\infty} g_{i,n}(t,\xi,x) = g_i(t,\xi,x).
  \end{align}

  Because $h_i$ is Lipschitz continuous \eqref{eq:h-Lip-assum} and $f_{i,n} = g_{i,n} + h_i$, for any $t \geq 0$, $\xi \in \bar{\mathcal{O}}$, $i \in \{1,....,r\}$, and $x,y \in \mathbb{R}^r$,
  \begin{align}
    &|f_{i,n}(t,\xi,x) - f_{i,n}(t,\xi,y)| \leq (2n + L(t)) |x-y|\\
    &|f_{i,n}(t,\xi,x)| \leq |f_{i}(t,\xi,x)| + 2|h_i(t,\xi,x)| \label{eq:f-bound}\\
    &f_{i,n}(t,\xi,x) - f_{i,n}(t,\xi,y) \leq L(t)|x-y| \text{ for } x>y,\\
    &\lim_{n \to +\infty} f_{i,n}(t,\xi,x) = f_i(t,\xi,x). \label{eq:f-limit}
  \end{align}

  Let $F_{n}: [0,T]\times E \to E$ be the Nemytskii operator for $(f_{1,n},...,f_{r,n})$
  \begin{equation}
    [F_n(t,x)]_i(\xi) = f_{i,n}(t,\xi,x(\xi)).
  \end{equation}
  Because the $f_{i,n}$ are each globally Lipschitz continuous, standard Picard iteration arguments show that there exists a unique solution $u_n \in E_T$ solving
  \[u_n(t) = \int_0^t S(t-s)F_n(s,u_n(s))ds + z(t).\]

  We prove some uniform bounds on the sequence $u_n$. Let $v_n(t) = u_n(t) - z(t)$. These $v_n$ are weakly differentiable and they solve the integral equation
  \[v_n(t) = \int_0^t S(t-s)F_n(s,v_n(s) + z(s))ds.\]
  The $v_n$ weakly solve the differential equation
  \[\frac{\partial v_{n,i}}{\partial t}(t,\xi) = \mathcal{A}_i v_{n,i}(t,\xi) + g_{i,n}(t,\xi,v_{n,i}(t,\xi) + z_i(t,\xi)) + h_i(t,\xi, v_n(t,\xi) + z(t,\xi)).\]
  Arguing as in Theorem 7.7 of \cite{dpz} and Proposition 6.2.2 of \cite{cerrai}, we can assume without loss of generality that $v_n$ are strongly differentiable.

  By Proposition \ref{prop:left-deriv} in the appendix, for $i_t \in \{1,...,r\}$, $\xi_t \in \mathcal{O}$, such that $v_{i_t}(t,\xi_t)\sgn(v_{i_t}(t,\xi_t)) = |v(t,\cdot)|_E$,
  \begin{align*}
    \frac{d^-}{dt} &|v_n(t)|_E\\
    \leq &\mathcal{A}_{i_t} v_{n,i_t}(t,\xi_t)\sgn(v_{n,i_t}(t,\xi_t)) \\
    &+ g_{i_t,n}(t,\xi_t, v_{n,i_t}(t,\xi_t) + z_{n,i_t}(t,\xi_t))\sgn(v_{n,i_t}(t,\xi_t)) \\
    &+ h_{i_t}(t,\xi_t, v_n(t,\xi_t) + z_n(t,\xi_t))\sgn(v_{n,i_t}(t,\xi_t)).
  \end{align*}
  Because $\mathcal{A}_{i_t}$ is a second-order elliptic differential operator (see Assumption \ref{assum:elliptic-operators}) and $i_t, \xi_t$ are a maximizer or minimizer, the concavity of a function at its maximum/minimum implies that
  \[\mathcal{A}_{i_t} v_{n,i_t}(t,\xi_t)\sgn(v_{n,i_t}(t,\xi_t)) \leq 0.\]
   Because $g_{i_t,n}(t,\xi_t,\cdot)$ is non-increasing and $h_{i_t}$ is Lipschitz continuous, by adding and subtracting $f_{i_t,n}(t,\xi_t,z(t,\xi_t))\sgn(v_{n,i_t}(t,\xi_t))$, we see that
  \begin{align*}
    \frac{d^-}{dt} &|v_n(t)|_E\\
    &\leq f_{i_t,n}(t,\xi_t, z(t,\xi_t))\sgn(v_{n,i_t}(t,\xi_t)) \\
    &\qquad+ (g_{i_t,n}(t,\xi_t, v_{n,i_t}(t,\xi_t)+z_{i_t}(t,\xi_t)) - g_{i_t,n}(t,\xi_t, z_{i_t}(t,\xi_t)))\sgn(v_{n,i_t}(t,\xi_t)) \\
    &\qquad+ |h_{i_t}(t,\xi_t, v_{n,i_t}(t,\xi_t)+z_{i_t}(t,\xi_t)) - h_{i_t}(t,\xi_t, z(t,\xi_t))|\\
    &\leq |f_{i_t,n}(t,\xi_t, z(t,\xi_t))| + L(t)|v_{i_t}(t,\xi_t)| \\
    &\leq \sup_n\sup_{s \in [0,t]} \sup_{i \in \{1,...,r\}} \sup_{\xi \in \mathcal{O}} |f_{i,n}(s,\xi,z(s,\xi))| + L(t)|v(t)|_E.
  \end{align*}
  By Gr\"onwall's inequality and \eqref{eq:f-bound},
  \[\sup_n \sup_{t \in [0,T]} |v_n(t)|_E \leq C_T \sup_n \sup_{s \in [0,t]} \sup_{i \in \{1,...,r\}} \sup_{\xi \in \mathcal{O}} |f_{i,n}(s,\xi,z(s,\xi))|<+\infty.\]
  Then
  \[\sup_n  |u_n|_{E_T} \leq \sup_n (|v_n|_{E_T} + |z|_{E_T})<+\infty. \]
  By \eqref{eq:f-bound} and \eqref{eq:h-lin-growth}
  $$\sup_{n}|F_n(\cdot, u_n(\cdot))|_{E_T} \leq |F(\cdot,u_n(\cdot))|_{E_T} + 2L(T)(1 + |u_n|_{E_T}) <+\infty$$
  By standard elliptic regularity arguments, because
  \[v_n(t) = \int_0^t S(t-s)F_n(s,u_n(s))ds\]
  and $\sup_{n}|F_n(\cdot, u_n(\cdot))|_{E_T} <+\infty$,
  there exist $\gamma>0$, $\beta>0$ such that
  \[\sup_n \sup_{i \in \{1,...,r\}}\sup_{\substack{s,t \in [0,T]\\s\not = t}} \sup_{\substack{\xi,\eta \in \mathcal{O}\\\xi \not =\eta}} \frac{|v_{n,i}(t,\xi) - v_{n,i}(s,\eta)|}{|t-s|^\gamma + |x-y|^\beta}<+\infty.\]

  By the Arzela-Ascoli Theorem, there exists a subsequence (relabeled $v_n$) and a limit $\tilde{v} \in E_T$ such that $v_n \to \tilde{v}$ in the $E_T$ norm.

  By the dominated convergence theorem and \eqref{eq:f-limit},
  \[\tilde{v}(t) = \int_0^t S(t-s)F(s,\tilde{v}(s) + z(s))ds.\]
  Then $\tilde{u}:= \tilde{v} + z$ is a solution to \eqref{eq:M-def}.
\end{proof}

\begin{theorem}[Lipschitz continuity of $\mathcal{M}$] \label{thm:M-Lipschitz}
  For any $T>0$, $\mathcal{M}$ is a Lipschitz continuous operator from $E_T \to E_T$. There exists $C_T>0$ depending only on $L(T)$ from Assumption \ref{assum:vector-field} such that for any $z_1,z_2 \in E_T$,
  \begin{equation} \label{eq:M-Lipschitz}
    |\mathcal{M}(z_1) - \mathcal{M}(z_2)|_{E_T} \leq C_T |z_1 - z_2|_{E_T}.
  \end{equation}
  In particular, this theorem proves that the solution to \eqref{eq:M-def} is unique.
\end{theorem}

\begin{proof}
  Let $z_1,z_2 \in E_T$. Let $u_1 := \mathcal{M}(z_1)$ and $u_2 := \mathcal{M}(z_2)$ be solutions to \eqref{eq:M-def}. Let $\tilde{u} = u_1-u_2$ and $\tilde z = z_1 - z_2$. Let $v_1 = u_1 - z_1$, $v_2 = u_2 - z_2$.  Let $\tilde v= v_1 - v_2$. Then by the definition \eqref{eq:M-def},
  \[\tilde v(t) = \int_0^t S(t-s) (F(s,v_1(s) + z_1(s)) - F(s,v_2(s) + z_2(s)))ds.\]

  Because $\tilde v$ is written as a convolution with a semigroup generated by an elliptic operator, $\tilde v$ is weakly differentiable and
  \[\frac{\partial \tilde v_i}{\partial t}(t,\xi) = \mathcal{A}_i \tilde v_i(t,\xi) + (f_i(t,\xi,v_1(t,\xi) + z_1(t,\xi))-f_i(t,\xi ,v_2(t,\xi) + z_2(t,\xi))).\]
  By arguing as in Theorem 7.7 of \cite{dpz} or Proposition 6.2.2 of \cite{cerrai}, we can assume without loss of generality that $\tilde{v}$ is strongly differentiable.

  By Proposition \ref{prop:left-deriv} in the Appendix, for any $t>0$ and any index $i_t \in \{1,...,r\}$ and $\xi_t \in \mathcal{O}$ such that
  \[|\tilde v(t)|_E = \tilde v_{i_t}(t,\xi_t)\sgn(\tilde v_{i_t}(t,\xi_t)),\]
  \begin{align*}
    &\frac{d^-}{dt} |\tilde v(t)|_E \\
    &\leq \mathcal{A}_{i_t} \tilde v_{i_t}(t,\xi_t) \sgn(\tilde v_{i_t}(t,\xi_t)) \\
    &\qquad+ (g_{i_t}(t,\xi_t,v_{1,i_t}(t,\xi_t) + z_{1,i_t}(t,\xi_t)) - g_{i_t}(t,\xi_t,v_{2,i_t}(t,\xi_t) + z_{2,i_t}(t,\xi_t)))\sgn(\tilde v_{i_t}(t,\xi_t))\\
    &\qquad+ (h_{i_t}(t,\xi_t,v_1(t,\xi_t) +z_1(t,\xi_t)) - h_{i_t}(t,\xi_t,v_2(t,\xi_t) + z_2(t,\xi_t)))\sgn(\tilde v_{i_t}(t,\xi_t)),
  \end{align*}
  where $g_i$ is the non-increasing function and $h_i$ is the Lipschitz continuous function from Assumption \ref{assum:vector-field}.

  By the ellipticity condition on $\mathcal{A}_{i_t}$ from Assumption \ref{assum:elliptic-operators}, because $\xi_t$ is a maximizer or minimizer of $\tilde{v}_{i_t}(t,\cdot)$, the concavity of a function at its maximum/minimum implies that
  \[\mathcal{A}_{i_t} \tilde{v}_{i_t}(t,\xi_t) \sgn(\tilde{v}_{i_t}(t,\xi_t)) \leq 0.\]

  For any $t>0$ there are two cases: either
  $$\sgn(\tilde{v}_{i_t}(t,\xi_t)) = \sgn(\tilde{v}_{i_t}(t,\xi_t) + \tilde{z}_{i_t}(t,\xi_t))$$
   or
  $$\sgn(\tilde{v}_{i_t}(t,\xi_t)) \not= \sgn(\tilde{v}_{i_t}(t,\xi_t) + \tilde{z}_{i_t}(t,\xi_t)).$$

  If $\sgn(\tilde{v}_{i_t}(t,\xi_t)) = \sgn(\tilde{v}_{i_t}(t,\xi_t) + \tilde{z}_{i_t}(t,\xi_t))$, then because $g_{i_t}$ is non-increasing \eqref{eq:g-decr},
  \[(g_{i_t}(t,\xi_t, v_{1,i_t}(t,\xi_t) + z_{1,i_t}(t,\xi_t)) - g_{i_t}(t,\xi_t, v_{2,i_t}(t,\xi_t) + z_{2,i_t}(t,\xi_t)))\sgn(\tilde v_{i_t}(t,\xi_t)) \leq 0.\]
  Due to the Lipschitz continuity of $h$ \eqref{eq:h-Lip-assum},
  \[\frac{d^-}{dt} |\tilde v(t)|_E \leq  L(t)|\tilde{v}(t)|_E + L(t)|\tilde{z}(t)|_E.\]

  On the other hand, if $\sgn(\tilde{v}_{i_t}(t,\xi_t)) \not= \sgn(\tilde{v}_{i_t}(t,\xi_t) + \tilde{z}_{i_t}(t,\xi_t))$, then
  \[|\tilde v_{i_t}(t,\xi_t)| \leq |\tilde z_{i_t}(t,\xi_t)|.\]
  Because $i_t$ and $\xi_t$ maximize $\tilde v$, and the $E$ norm is a supremum norm, in the case where $\sgn(\tilde{v}_{i_t}(t,\xi_t)) \not= \sgn(\tilde{v}_{i_t}(t,\xi_t) + \tilde{z}_{i_t}(t,\xi_t))$,
  \[|\tilde v(t)|_E = |\tilde v_{i_t}(t,\xi_t)| \leq |\tilde z_{i_t}(t,\xi_t)| \leq |\tilde z(t)|_E.\]

  We have shown that for any given $t>0$ there are only two possibilities. For any $t>0$, either
  \[\frac{d^-}{dt}|\tilde v(t)|_E \leq L(t) |\tilde{v}(t)|_E + L(t) |\tilde{z}(t)|_E,\]
  or
  \[|\tilde v(t)|_E \leq |\tilde z(t)|_E.\]
  For $t \in [0,T]$, let $\phi(t): = \max\{ |\tilde z|_{E_T}, |\tilde v(t)|_E\}$. Note that because $|\tilde{v}(0)|_E=0$ it follows that $\phi(0) = |\tilde z|_{E_T}$. Therefore,
  \begin{align*}
    \phi(t) &\leq |\tilde{z}|_{E_T} + \int_0^t \frac{d^{-}}{ds}|\tilde{v}(s)|_E \mathbbm{1}_{\{|\tilde{v}(s)|_E> |\tilde{z}|_{E_T}\}}ds\\
    &\leq C_T |\tilde{z}|_{E_T} + L(T)\int_0^t \left(|\tilde{v}(s)|_E + |\tilde{z}(s)|_E\right)ds\\
    &\leq C_T  |\tilde{z}|_{E_T} +2 L(T)\int_0^t \phi(s)ds.
  \end{align*}
  By Gr\"onwall's inequality, there exists $C_T>0$ such that
  \[\sup_{t \in [0,T]}\phi(t) \leq C_T |\tilde{z}|_{E_T}.\]
  Therefore
  \[|\tilde{v}|_{E_T}\leq C_T |\tilde{z}|_{E_T}.\]
  Because $\tilde u(t) = \tilde v(t) + \tilde z(t)$,
  \[|\tilde u|_{E_T} \leq (C_T + 1) |\tilde z|_{E_T},\]
  proving our result.
\end{proof}

\section{Existence and uniqueness of the solution to controlled stochastic reaction diffusion equations} \label{S:exit-uniq}
In this section we prove that under Assumptions \ref{assum:vector-field}, \ref{assum:sigma}, \ref{assum:elliptic-operators}, and \ref{assum:noise}, the solutions to the controlled SPDE \eqref{eq:controlled-mild} exist and are unique. Because our assumptions are weaker than previous results, these existence and uniqueness results cannot be found in the literature.  The existence of the mild solutions to the uncontrolled SPDE \eqref{eq:mild-def} is a corollary obtained by using the trivial control $u\equiv 0$. Notice that our assumptions are strictly weaker than those in \cite{c-2003} or \cite{cr-2004}.

\begin{theorem} \label{thm:exist-uniq}
  For any $x \in E$,  $N>0$, $u \in \mathscr{A}_N$, and $\e>0$, there exists a unique solution $X^{\e,u}_x$ to \eqref{eq:controlled-mild} and the solution is $E_T$ valued.
\end{theorem}

\begin{proof}
  We build a contraction mapping. Let $\hat{E}_T$ denote the collection of continuous random fields $\psi: \Omega \times [0,T] \times \bar{\mathcal{O}}\times\{1,...,r\} \to \mathbb{R}$ that are adapted to the filtration $\mathcal{F}_t$.

  By the definition of $\mathcal{M}$ \eqref{eq:M-def}, $X^{\e,u}_x$ is a solution to \eqref{eq:controlled-mild} if and only if it satisfies
  \begin{equation} \label{eq:X-M-def}
    X^{\e,u}_x = \mathcal{M}\big(S(\cdot) x + Y^u(X^{\e,u}_x) + \sqrt{\e} Z(X^{\e,u}_x)\big)
  \end{equation}
  where for any $\psi \in \hat{E}_T$, $Y^u(\psi) \in \hat{E}_T$ and $Z(\psi) \in \hat{E}_T$ are defined by
  \begin{equation} \label{eq:Y-mild-def}
    Y^u(\psi)(t) = \int_0^t S(t-s)R(s,\psi(s))Qu(s)ds
  \end{equation}
  and
  \begin{equation} \label{eq:Z-mild-def}
    Z(\psi)(t) = \int_0^t S(t-s)R(s,\psi(s))dw(s).
  \end{equation}

  Let $\mathcal{K}^{\e,u}_x:\hat E_T \to \hat E_T$ be defined by
  \[\mathcal{K}^{\e,u}_x(\psi) = \mathcal{M} \left(S(\cdot) x + Y^u(\psi) + \sqrt{\e}Z(\psi)\right).\]

  Let $\beta, \rho>0$ be the constants from Assumption \ref{assum:noise}. Let $\alpha \in \left(0, \frac{1}{2}\left(1 - \frac{\beta(\rho-2)}{\rho} \right) \right)$, $\gamma \in (0,\alpha)$, and $p> \max\left\{\frac{1}{\alpha-\gamma}, \frac{d}{\gamma} \right\}$.
  By Theorem \ref{thm:M-Lipschitz}, there exists a constant $C_T>0$ such that for $\psi_1, \psi_2 \in L^p(\Omega: E_T)$,
  \begin{align} \label{eq:K-contr-bound}
    &\E|\mathcal{K}^{\e,u}_x(\psi_1) - \mathcal{K}^{\e,u}_x(\psi_2)|_{E_T}^p \nonumber\\
    &\leq C_T \left(\E|Y^u(\psi_1) - Y^u(\psi_2)|_{E_T}^p + \e^{\frac{p}{2}}\E|Z(\psi_1) - Z(\psi_2)|_{E_T}^p \right).
  \end{align}
  By Theorem \ref{thm:Z-bound}  in the Appendix,
  \begin{align*}
    &\E|Z(\psi_1) - Z(\psi_2)|_{E_T}^p \\
    &\leq C \E \int_0^T \left(\int_0^t (t-s)^{-2\alpha - \frac{\beta(\rho-2)}{\rho}} \max_{i \in \{1,...,r\}}|R_{i \cdot}(s,\psi_1(s))-R_{i \cdot}(s,\psi_2(s))|_{E}^2ds \right)^{\frac{p}{2}}dt.
  \end{align*}
  By the Lipschitz continuity of $R$, and the fact that $-2\alpha - \frac{\beta(\rho-2)}{\rho}>-1$,
  \begin{align} \label{eq:Z-psi-bound}
    &\E|Z(\psi_1) - Z(\psi_2)|_{E_T}^p \leq C_T \int_0^T \E|\psi_1 - \psi_2|_{E_t}^pdt.
  \end{align}
  Similarly, by Theorem \ref{thm:control-Y-bounds}, because $u \in \mathscr{A}_N$,
  \begin{equation} \label{eq:Y-psi-bound}
    \E\left|Y^u(\psi_1) - Y^u(\psi_2)\right|_{E_T}^p \leq C_T N^{\frac{p}{2}} \int_0^T \E |\psi_1 - \psi_2|_{E_t}^pdt.
  \end{equation}
  By \eqref{eq:K-contr-bound}, \eqref{eq:Z-psi-bound}, and \eqref{eq:Y-psi-bound},
  \begin{align*}
    &\E|\mathcal{K}^{\e,u}_x(\psi_1) - \mathcal{K}^{\e,u}_x(\psi_2)|^p_{E_T}
    \leq C_T \left( \e^{\frac{p}{2}} + N^{\frac{p}{2}} \right) \int_0^T |\psi_1 - \psi_2|_{E_t}^p dt.
  \end{align*}

  There exists a $T_0$ small enough so that $C_{T_0}T_0\left( \e^{\frac{p}{2}} + N^{\frac{p}{2}} \right)<1$. Then $\mathcal{K}^{\e,u}_x$ is a contraction mapping on $L^p(\Omega:E_{T_0})$ and there exists a unique fixed point $X^{\e,u}_x$ solving \eqref{eq:X-M-def} for $t \in [0,T_0]$. This argument can be repeated on $[T_0,2T_0]$, $[2T_0,3T_0]$ and so forth to prove that there exists a unique global solution to the control equation \eqref{eq:X-M-def}.
\end{proof}

\begin{corollary} \label{cor:exist-uniq}
  For any $x \in E$ and $\e>0$, there exists a unique global mild solution to the uncontrolled SPDE $X^\e_x$ \eqref{eq:mild-def}.
\end{corollary}
\begin{proof}
  This is immediate by using the trivial control $u\equiv 0$ in Theorem \ref{thm:exist-uniq} because $X^\e_x = X^{\e,0}_x$.
\end{proof}

Next we prove that the solutions to the control equation \eqref{eq:X-M-def} are bounded in $L^p(\Omega:E_T)$ uniformly for $u \in \mathscr{A}_N$, and bounded subsets of $\e>0$ and bounded subsets of $x \in E$.

\begin{theorem} \label{thm:control-a-priori}
  For $T>0$ and $p>1$, there exists $C_{T,p}>0$ such that for any $N>0$, $u \in \mathscr{A}_N$, $\e>0$, and $x \in E$,
  \begin{equation} \label{eq:control-a-priori}
    \E \left|X^{\e,u}_x\right|_{E_T}^p \leq C_{T,p} e^{C_{T,p}(\e^{\frac{p}{2}} + N^{\frac{p}{2}})}  \left( 1 +|x|_E^p\right).
  \end{equation}
\end{theorem}

\begin{proof}
  By \eqref{eq:X-M-def} and the Lipschitz continuity of $\mathcal{M}$, Theorem \ref{thm:M-Lipschitz},
  \begin{align*}
    &\E|X^{\e,u}_x|_{E_T}^p = \E\left|\mathcal{M}\left( S(\cdot)x + Y^u(X^{\e,u}_x) + \sqrt{\e}Z(X^{\e,u}_x) \right)\right|^p_{E_T}\\
    &\leq  C_p\E \left|\mathcal{M}\left( S(\cdot)x + Y^u(X^{\e,u}_x) + \sqrt{\e}Z(X^{\e,u}_x) \right) - \mathcal{M}(0) \right|_{E_T}^p + C_p|\mathcal{M}(0)|_{E_T}^p\\
    &\leq C_{T,p} \left( 1 + |x|_E^p + \E|Y^u(X^{\e,u}_x)|_{E_T}^p + \e^{\frac{p}{2}}\E\left|Z(X^{\e,u}_x) \right|_{E_T}^p \right).
  \end{align*}
  By Theorem \ref{thm:Z-bound}, Theorem \ref{thm:control-Y-bounds} and the fact that $R$ has linear growth,
  for large enough $p$,
  \begin{align*}
    &\E|X^{\e,u}_x|_{E_T}^p
    \leq C_{T,p} \left( 1 +|x|_E+ \left(\e^{\frac{p}{2}} + N^{\frac{p}{2}} \right)\int_0^T \E|X^{\e,u}_x|_{E_t}^pdt\right).
  \end{align*}
  The result follows by Gr\"onwall's inequality.
\end{proof}

\begin{corollary} \label{cor:SPDE-a-priori}
  For $T>0$ and $p>1$, there exists $C_{T,p}>0$ such that for any $\e>0$, and $x \in E$,
  \begin{equation} \label{eq:SPDE-a-priori}
    \E \left|X^{\e}_x\right|_{E_T}^p \leq C_{T,p} e^{C_{T,p}\e^{\frac{p}{2}} } \left( 1 +|x|_E^p\right).
  \end{equation}
\end{corollary}

\begin{proof}
  This is an immediate consequence of Theorem \ref{thm:control-a-priori} and the fact that $X^\e_x = X^{\e,0}_x$.
\end{proof}

\section{Uniform large deviations principle over bounded subsets of $E$ -- Proof of Theorem \ref{thm:ULDP-bounded-subsets}} \label{S:ULDP-bounded-x}
In this section, we use Corollary \ref{cor:ULDP-suff-cond} to prove that the mild solutions $\{X^\e_x\}_{\substack{\e>0\\x \in E}}$ to \eqref{eq:intro-SPDE} satisfy a uniform large deviations principle that is uniform over bounded subsets of $E$.
%
%
%

\begin{proof}[Proof of Theorem \ref{thm:ULDP-bounded-subsets}]
  By Corollary \ref{cor:ULDP-suff-cond}, it is sufficient to prove that for any $K>0$, $N>0$, and $\delta>0$,
  \begin{equation}
    \lim_{\e \to 0} \sup_{|x|_E \leq K} \sup_{u \in \mathscr{A}_N} \Pro \left( \left|X^{\e,u}_x - X^{0,u}_x \right|_{E_T}>\delta \right) = 0.
  \end{equation}
  Let
  \begin{equation} \label{eq:Y-e-u}
    Y^{\e,u}_{x}(t) = \int_0^t S(t-s) R(s,X^{\e,u}_x(s))Q u(s)ds
  \end{equation}
  and
  \begin{equation} \label{eq:Z-e-u}
    Z^{\e,u}_{x}(t) = \int_0^t S(t-s)R(s,X^{\e,u}_x(s))dw(s).
  \end{equation}
  Using this notation,
  \[X^{\e,u}_x = \mathcal{M}\left(S(\cdot) x + Y^{\e,u}_x + \sqrt{\e} Z^{\e,u}_x\right),\]
  where $\mathcal{M}:E_T \to E_T$ solves \eqref{eq:M-def}.

  By the Lipschitz continuity of $\mathcal{M}$ (Theorem \ref{thm:M-Lipschitz}),
  \begin{align} \label{eq:X-diff-bounded}
    & |X^{\e,u}_{x} - X^{0,u}_{x}|_{E_T} 
    \leq C_T |Y^{\e,u}_{x} - Y^{0,u}_{x}|_{E_T} + C_T \sqrt{\e} |Z^{\e,u}_{x}|_{E_T}.
  \end{align}
  By Theorem \ref{thm:control-Y-bounds}, for $u \in \mathcal{A}_N$, $\e>0$, and $x \in {E}$,
  \begin{align*}
    &|Y^{\e,u}_{x} - Y^{0,u}_{x}|_{E_T}\\
    & \leq CN^{\frac{1}{2}} \sup_{t \in [0,T]}\left(\int_0^t (t-s)^{-\frac{\beta(\rho-2)}{\rho}} \max_{i \in \{1,...,r\}} |R_{i \cdot}(s,X^{\e,u}_{x}(s)) - R_{i \cdot}(s,X^{0,u}_{x}(s))|^2_{E} ds  \right)^{\frac{1}{2}}.
  \end{align*}
  By the Lipschitz continuity of $R$ (Assumption \ref{assum:sigma}),
  \begin{align*}
    &|Y^{\e,u}_{x} - Y^{0,u}_{x}|_{E_T}\\
    &\leq C N^{\frac{1}{2}}\sup_{t \in [0,T]}\left(\int_0^t (t-s)^{-\frac{\beta(\rho-2)}{\rho}}  |X^{\e,u}_{x}(s) - X^{0,u}_{x}(s)|_E^2 ds  \right)^{\frac{1}{2}}.
  \end{align*}
  By Assumption \eqref{eq:beta-rho-relation}, $\frac{\beta(\rho-2)}{\rho}<1$. 
  For $p>\frac{2}{1 - \frac{\beta(\rho-2)}{\rho}}$, the H\"older inequality shows that
  \begin{align} \label{eq:Y-diff-bound}
    &|Y^{\e,u}_{x} - Y^{0,u}_{x}|_{E_T}^p 
    \leq C_{p,T}N^{\frac{p}{2}} \int_0^T |X^{\e,u}_{x} - X^{0,u}_{x}|_{E_t}^p dt
  \end{align}

  Let $\alpha \in \left(0, \frac{1}{2}\left(1 - \frac{\beta(\rho-2)}{\rho} \right) \right)$, $\gamma \in (0,\alpha)$, and $p> \max\left\{\frac{1}{\alpha-\gamma}, \frac{d}{\gamma} \right\}$. By Theorem \ref{thm:Z-bound},
  \begin{align}
    &\E|Z^{\e,u}_x|_{E_T}^p \nonumber\\
    &\leq C_{T,p} \E \int_0^T \left(\int_0^t (t-s)^{-2\alpha - \frac{\beta(\rho-2)}{\rho}} \max_{i \in \{1,...r\}}|R_{i \cdot}(s,X^{\e,u}_x(s))|_{E}^2ds \right)^{\frac{p}{2}}dt.
  \end{align}
  By the linear growth of $R$, and the fact that $-2\alpha - \frac{\beta(\rho-2)}{\rho}>-1$,
  \begin{align*}
    &\E|Z^{\e,u}_x|_{E_T}^p \leq C_{T,p} \left( 1 +  \E|X^{\e,u}_x|_{E_T}^p\right).
  \end{align*}
  By \eqref{eq:control-a-priori},
  \begin{align} \label{eq:Z-control-bound}
    &\E|Z^{\e,u}_x|_{E_T}^p \leq C_{T,p} e^{\left( \e^{\frac{p}{2}} + N^{\frac{p}{2}} \right)C_{T,p}} \left(1 + |x|_E^p \right).
  \end{align}

  Therefore by \eqref{eq:X-diff-bounded}, \eqref{eq:Y-diff-bound}, and \eqref{eq:Z-control-bound},
  \begin{align*}
    &\E |X^{\e,u}_x - X^{0,u}_x|_{E_T}^p \\
     &\leq C_{T,p}N^{\frac{p}{2}} \int_0^T \E|X^{\e,u}_{x} - X^{0,u}_{x}|_{E_t}^p dt + C_{T,p} \e^{\frac{p}{2}}e^{\left( \e^{\frac{p}{2}} + N^{\frac{p}{2}} \right)C_{T,p}} \left(1 + |x|_E^p \right).
  \end{align*}
  By Gr\"onwall's inequality, for any $K>0$,
  \[\sup_{|x|_E \leq K}\sup_{u \in \mathcal{A}_N}\E |X^{\e,u}_x - X^{0,u}_x|_{E_T}^p \leq \e^{\frac{p}{2}} C_{T,p}e^{\left( \e^{\frac{p}{2}} + N^{\frac{p}{2}} \right)C_{T,p}} \left( 1 + K^p\right).\]
  By the Chebyshev inequality,
  \[\lim_{\e \to 0} \sup_{|x|_E \leq K}\sup_{u \in \mathcal{A}_N}\Pro \left( |X^{\e,u}_x - X^{0,u}_x|_{E_T}>\delta \right) =
    0.\]
    Then Theorem \ref{thm:ULDP-bounded-subsets} is a consequence of Corollary \ref{cor:ULDP-suff-cond}.
\end{proof}

\section{Uniform large deviations when $\sigma$ is uniformly bounded -- Proof of Theorem \ref{thm:ULDP-sigma-bounded}} \label{S:sigma-bounded}

\begin{proof}[Proof of Theorem \ref{thm:ULDP-sigma-bounded}]
  By Corollary \ref{cor:ULDP-suff-cond} it suffices to show that for any $\delta>0$ and $N>0$,
  \begin{align*}
    \lim_{\e \to 0} \sup_{x \in E} \sup_{u \in \mathcal{A}_N} \Pro \left( |X^{\e,u}_x - X^{0,u}_x|_{E_T}>\delta\right)=0.
  \end{align*}

  Let $Y^{\e,u}_x$ and $Z^{\e,u}_x$ be the solutions to \eqref{eq:Y-e-u} and \eqref{eq:Z-e-u}. Then
  \[X^{\e,u}_x = \mathcal{M}(S(\cdot) x + Y^{\e,u}_x + \sqrt{\e} Z^{\e,u}_x).\]
  By the Lipschitz continuity of $\mathcal{M}$ (Theorem \ref{thm:M-Lipschitz}), \eqref{eq:X-diff-bounded} holds.
  By Theorem \ref{thm:control-Y-bounds}, \eqref{eq:Y-diff-bound} holds.
  By Theorem \ref{thm:Z-bound}, for large enough $p>\frac{2}{1 - \frac{\beta(\rho-2)}{\rho}}$, there exists $C_{T,p}>0$ such that
  \begin{align}
    &\E|Z^{\e,u}_x|_{E_T}^p \nonumber\\
    &\leq C_{T,p} \E \int_0^T \left(\int_0^t (t-s)^{-2\alpha - \frac{\beta(\rho-2)}{\rho}} \max_{i \in \{1,...r\}}|R_{i \cdot}(s,X^{\e,u}_x(s))|_{E}^2ds \right)^{\frac{p}{2}}dt.
  \end{align}
  By \eqref{eq:sigma-bounded},
  \[\sup_{s \in [0,T]} \sup_{ x \in E} \sup_{n \in \{1,...r\}}|R_{\cdot n}(s,X)|_E \leq L(T).\]
  Because $-2\alpha - \frac{\beta(\rho-2)}{\rho}>-1$,
  \begin{equation}  \label{eq:Z-unif-bound}
    \sup_{\e \in (0,1)}\sup_{x \in E} \sup_{u \in \mathcal{A}_N}\E|Z^{\e,u}_x|^p_{E_T} \leq C_{T,p}.
  \end{equation}

  By \eqref{eq:X-diff-bounded}, \eqref{eq:Y-diff-bound}, and \eqref{eq:Z-unif-bound}, there exists $C_{T,p}>0$ such that for any $x \in E$, $\e>0$, $N>0$, and $u \in \mathscr{A}_N$,
  \begin{equation}
    \E|X^{\e,u}_x - X^{0,u}_x|_{E_T}^p \leq C_{T,p} N^{\frac{p}{2}} \int_0^T \E|X^{\e,u}_x - X^{0,u}_x|_{E_t}^pdt + C_{T,p} \e^{\frac{p}{2}}.
  \end{equation}
  By Gr\"onwall's inequality,
  \begin{equation}
    \E|X^{\e,u}_x - X^{0,u}_x|_{E_T}^p \leq C_{T,p} \e^{\frac{p}{2}} e^{C_{p,T}N^{\frac{p}{2}} T}.
  \end{equation}
  This estimate is uniform with respect to $x \in E$ and therefore
  \begin{equation}
    \lim_{\e \to 0} \sup_{x \in E} \sup_{u \in \mathscr{A}_N}\E|X^{\e,u}_x- X^{0,u}_x|_{E_T}^p = 0.
  \end{equation}
  By the Chebyshev inequality,
  \[\lim_{\e \to 0} \sup_{x \in E}\sup_{u \in \mathcal{A}_N}\Pro \left( |X^{\e,u}_x - X^{0,u}_x|_{E_T}>\delta \right) =
    0.\]
    Then Theorem \ref{thm:ULDP-sigma-bounded} is a consequence of Corollary \ref{cor:ULDP-suff-cond}.
\end{proof}

\section{Uniform large deviations when $f$ has super-linear dissipativity -- Proof of Theorem \ref{thm:ULDP-super-dissip}} \label{S:super-dissip}

\begin{proof}[Proof of Theorem \ref{thm:ULDP-super-dissip}]
By Corollary \ref{cor:ULDP-suff-cond}, it is sufficient to show that for any $\delta>0$ and $N >0$,
\[\lim_{\e \to 0} \sup_{x \in E} \sup_{u \in \mathscr{A}_N} \Pro\left(\left|X^{\e,u}_x - X^{0,u}_x \right|_{E_T}>\delta \right)=0.\]

For $x \in E$, define $\mathcal{M}_x: E_T \to E_T$ by
\begin{equation} \label{eq:M_x-def-not-appendix}
  \mathcal{M}_x(\varphi) := \mathcal{M}(S(\cdot)x + \varphi).
\end{equation}
Under the super-linear dissipativity assumption (Assumption \ref{assum:super-dissip}), $\mathcal{M}_x$ satisfies certain bounds that are independent of the initial condition $x$. These results are presented in Appendix \ref{S:unif-bounds}.

We observe that $X^{\e,u}_x$ can be written as
\begin{equation}
  X^{\e,u}_x = \mathcal{M}_x(Y^{\e,u}_x + \sqrt{\e} Z^{\e,u}_x)
\end{equation}
where
\begin{equation}
  Y^{\e,u}_{x}(t) = \int_0^t S(t-s) R(s,X^{\e,u}_x(s))Qu(s)ds
\end{equation}
\begin{equation}
  Z^{\e,u}_{x}(t) = \int_0^t S(t-s) R(s,X^{\e,u}_x(s))dw(s)
\end{equation}

By Theorem \ref{thm:Z-bound}, for $u \in \mathscr{A}_N$, $\e>0$, $x \in E$, and any $\alpha \in \left(0, \frac{1}{2} \left( 1 - \frac{\beta(\rho-2)}{\rho} \right) \right)$, $\gamma \in (0,\alpha)$, and $p> \max\left\{\frac{1}{\alpha - \gamma}, \frac{d}{\gamma} \right\}$,
\begin{align*}
  &\E|Z^{\e,u}_x|_{E_T}^p \\
  &\leq C_{T,p} \E \int_0^T\left( \int_0^t (t-s)^{-2\alpha - \frac{\beta(\rho-2)}{\rho}}  \max_{i \in \{1,...,r\}} |R_{i \cdot}(s,X^{\e,u}_x(s))|_E^2ds\right)^{\frac{p}{2}}dt.
\end{align*}
By the assumed growth rate on $\sigma$ (and therefore $R$) in Assumption \ref{assum:super-dissip},
\begin{align*}
  &\E|Z^{\e,u}_x|_{E_T}^p \\
  &\leq C_{T,p} \int_0^T\E\left( \int_0^t (t-s)^{-2\alpha - \frac{\beta(\rho-2)}{\rho}}  (1 + |X^{\e,u}_x(s)|_E^{2\nu})ds\right)^{\frac{p}{2}}dt.
\end{align*}
By the fact that $X^\e_x=\mathcal{M}_x(Y^\e_x + Z^\e_x)$ and \eqref{eq:M-bound-super-dissip}, 
\begin{align}
  |X^\e_x(t)|_E &\leq  C_t \left(1 + t^{-\frac{1}{m-1}} + |Y^{\e,u}_x + \sqrt{\e}Z^{\e,u}_x|_{E_t}\right)
\end{align}

Therefore,
\begin{align*}
  &\E|Z^{\e,u}_x|_{E_T}^p \\
  &\leq C_{T,p}\int_0^T\E\left( \int_0^t (t-s)^{-2\alpha - \frac{\beta(\rho-2)}{\rho}}  \left(1 + s^{-\frac{2\nu}{m-1}} + |Y^{\e,u}_x + \sqrt{\e}Z^{\e,u}_x|_{E_s}^{2\nu}\right)ds\right)^{\frac{p}{2}}dt.
\end{align*}
By \eqref{eq:nu} we can choose
\[\alpha:= \frac{1}{2} \left(1 - \frac{\beta(\rho-2)}{\rho} - \frac{2\nu}{m-1}  \right) \in \left(0, \frac{1}{2}\right).\]
Then for any $t>0$, by the properties of the Beta function, for any $t>0$,
\begin{align} \label{eq:Beta-function}
  &\int_0^t (t-s)^{-2\alpha - \frac{\beta(\rho-2)}{\rho}}   s^{-\frac{2\nu}{m-1}}ds \nonumber\\
  &=\int_0^1 (1-s)^{-2\alpha - \frac{\beta(\rho-2)}{\rho}}   s^{-\frac{2\nu}{m-1}}ds \nonumber\\
  &= \frac{\pi}{\sin\left(\frac{2\nu \pi}{m-1} \right)}.
\end{align}
Furthermore, by Assumption \ref{assum:super-dissip}, $\nu \in [0,1]$.
There will exist large enough constants such that
\begin{equation} \label{eq:Z-bound-a-p}
\E|Z^{\e,u}_x|_{E_T}^p \leq C_{T,p} \left( 1 +  \int_0^T  \E|Y^{\e,u}_x + \sqrt{\e}Z^{\e,u}_x|_{E_t}^p dt\right).
\end{equation}

Similarly, by Theorem \ref{thm:control-Y-bounds} and \eqref{eq:M-bound-super-dissip}
\begin{align*}
  |Y^{\e,u}_x|_{E_T} \leq &C_T N^{\frac{1}{2}}\sup_{t \in [0,T]} \int_0^t (t-s)^{-\frac{\beta(\rho-2)}{\rho}} \left(1 + s^{-\frac{2\nu}{m-1}} + |Y^{\e,u}_x + \sqrt{\e} Z^{\e,u}_x|_{E_s}^2 \right)ds \\
  &\leq C_T N^{\frac{1}{2}} \left( 1 + \sup_{t \in [0,T]} \int_0^t (t-s)^{-\frac{\beta(\rho-2)}{\rho}} (|Y^{\e,u}_x + \sqrt{\e} Z^{\e,u}_x|_{E_s}^2 )ds\right).
\end{align*}
By a H\"older inequality,
\begin{equation} \label{eq:Y-bound-a-p}
  |Y^{\e,u}_x|_{E_T}^p \leq C_{T,p} N^{\frac{p}{2}} \left( 1+  \int_0^T |Y^{\e,u}_x + \sqrt{\e} Z^{\e,u}_x|_{E_t}^p dt \right).
\end{equation}
Combining \eqref{eq:Z-bound-a-p} and \eqref{eq:Y-bound-a-p},
\begin{equation*}
  \E|Y^{\e,u}_x + \sqrt{\e} Z^{\e,u}_x|_{E_T}^p \leq  C_{T,p} \left(N^{\frac{p}{2}} + \e^{\frac{p}{2}} \right)\left( 1+  \int_0^T |Y^{\e,u}_x + \sqrt{\e} Z^{\e,u}_x|_{E_t}^p dt \right).
\end{equation*}

By Gr\"onwall's inequality,
\begin{equation}  \label{eq:Y+Z-bound}
  \E|Y^{\e,u}_x + \sqrt{\e}Z^{\e,u}_x|_{E_T}^p \leq C_{T,p}( N^{\frac{p}{2}} + \e^{\frac{p}{2}} ) e^{C_{T,p} ( N^{\frac{p}{2}} + \e^{\frac{p}{2}} )}.
\end{equation}

From Theorem \ref{thm:Z-bound},  \eqref{eq:Beta-function}, and \eqref{eq:Y+Z-bound} we can conclude that
\begin{align} \label{eq:Z-eps-bound-super-dissip}
  &\e^{\frac{p}{2}}\E |Z^{\e,u}_x|_{E_T}^p \nonumber\\
  &\leq C_{T,p}\e^{\frac{p}{2}}  \int_0^T\E\left( \int_0^t (t-s)^{-2\alpha - \frac{\beta(\rho-2)}{\rho}}  \left(1 + s^{-\frac{2\nu}{m-1}} + |Y^{\e,u}_x + \sqrt{\e}Z^{\e,u}_x|_{E_s}^{2\nu}\right)ds\right)^{\frac{p}{2}}dt\nonumber \\
  &\leq \e^{\frac{p}{2}}C_{T,p}(1 +  N^{\frac{p}{2}} + \e^{\frac{p}{2}} ) e^{C_{T,p} ( N^{\frac{p}{2}} + \e^{\frac{p}{2}} )}.
\end{align}
The above bound is uniform over $x \in E$.

The remainder of the proof is very similar to the proofs of Theorems \ref{thm:ULDP-bounded-subsets} and \ref{thm:ULDP-sigma-bounded}. By the Lipschitz continuity of $\mathcal{M}$ (Theorem \ref{thm:M-Lipschitz}),
\begin{align*}
    & |X^{\e,u}_{x} - X^{0,u}_{x}|_{E_T}
    \leq C_T |Y^{\e,u}_{x} - Y^{0,u}_{x}|_{E_T} + C_T \sqrt{\e} |Z^{\e,u}_{x}|_{E_T}
  \end{align*}
By \eqref{eq:Y-diff-bound},
\begin{align*}
    &|Y^{\e,u}_{x} - Y^{0,u}_{x}|_{E_T}^p
    \leq  C_{T,p}N^{\frac{p}{2}} \int_0^T |X^{\e,u}_{x} - X^{0,u}_{x}|_{E_t}^p dt
  \end{align*}
Therefore, \eqref{eq:Z-eps-bound-super-dissip} implies
  \begin{align*}
    &\E |X^{\e,u}_x - X^{0,u}_x|_{E_T}^p \\
    &\leq C_{T,p}N^{\frac{p}{2}} \int_0^T \E|X^{\e,u}_{x} - X^{0,u}_{x}|_{E_t}^p dt + \e^{\frac{p}{2}}C_{T,p}( N^{\frac{p}{2}} + \e^{\frac{p}{2}} ) e^{C_{T,p} ( N^{\frac{p}{2}} + \e^{\frac{p}{2}} )}  .
  \end{align*}
  By Gr\"onwall's inequality, there exists $C_{N,T,p}>0$ such that for all $\e \in (0,1)$,
  \[\sup_{x \in E}\sup_{u \in \mathcal{A}_N}\E |X^{\e,u}_x - X^{0,u}_x|_{E_T}^p \leq C_{N,T,p} \e^{\frac{p}{2}}\]
  By the Chebyshev inequality,
  \[\lim_{\e \to 0} \sup_{x \in E}\sup_{u \in \mathcal{A}_N}\Pro \left( |X^{\e,u}_x - X^{0,u}_x|_{E_T}>\delta \right) = 0.\]
  Then Theorem \ref{thm:ULDP-super-dissip} is a consequence of Corollary \ref{cor:ULDP-suff-cond}.
\end{proof}


\begin{appendix}
\section{The left derivative of the supremum norm} \label{S:appendix-subdiff}
Let $E = \left\{x \in  C(\bar{\mathcal{O}}\times \{1,...,r\}): x(\xi) = 0 \text{ for } \xi \in \partial \mathcal{O} \right\}$ endowed with the supremum norm
\[|x|_E := \sup_{\xi \in \bar{\mathcal{O}}}\sup_{i \in \{1,...,r\}} |x_i(\xi)|.\]

%
%

For a real-valued function $\phi: [0,T] \to \mathbb{R}$, define the left derivative by
\[\frac{d^-\phi}{dt}(t) = \limsup_{h \downarrow 0} \frac{\phi(t)- \phi(t-h)}{h}.\]

Let $\psi: [0,T]\times \bar{\mathcal{O}}\times \{1,...r\} \to \mathbb{R}$ be differentiable in its first argument.

\begin{proposition}[Proposition D.4 of \cite{dpz}] \label{prop:left-deriv}
  Assume that $\psi: [0,T]\times \bar{\mathcal{O}}\times \{1,...r\}$ has a continuous partial derivative in time.
  The left-derivative of the $E$ norm is bounded above by
  \begin{equation}
    \frac{d^-}{dt} |\psi_\cdot(t,\cdot)|_E \leq \left[  \frac{\partial \psi}{\partial t}\right]_{i_t}(t,\xi_t) \sgn(\psi_{i_t}(t,\xi_t))
  \end{equation}
  for any maximizer/minimizer $(i_t, \xi_t) \in \{1,...,r\}\times \bar{\mathcal{O}}$ such that
  \begin{equation}
    |\psi_{i_t}(t,\xi_t)|=|\psi_\cdot(t,\cdot)|_E.
  \end{equation}
\end{proposition}
\begin{proof}
  Fix $t >0$. Let $(i_t, \xi_t) \in \{1,....r\}\times \bar{\mathcal{O}}$ be a maximizer/minimizer of $\psi_\cdot(t,\cdot)$ such that
  \[|\psi_{i_t}(t,\xi_t)|=|\psi|_E.\]
  Notice that for another time $h\in(0,t)$,
  \[\psi_{i_t}(t-h,\xi_t)\sgn(\psi_{i_t}(t,\xi_t)) \leq |\psi_\cdot(t-h,\cdot)|_E.\]
  The left-derivative of $|\psi(t)|_E$ is
  \begin{align*}
    &\frac{d^-}{dt} |\psi(t)|_E\\
    &= \limsup_{h \downarrow 0} \frac{|\psi(t)|_E - |\psi(t+h)|_E}{h}\\
    &\leq \limsup_{h \downarrow 0} \frac{\left(\psi_{i_t}(t,\xi_t) - \psi_{i_t}(t-h,\xi_t) \right)\sgn(\psi_{i_t}(t,\xi_t))}{h}\\
    & = \left[\frac{\partial \psi}{\partial t}\right]_{i_t}\left(t,\xi_t \right) \sgn(\psi_{i_t}(t,\xi_t)).
  \end{align*}
\end{proof}

\section{Continuity in time and space of the stochastic convolution } \label{S:stoch-conv}
This appendix collects some results from \cite{c-2003,c-2009-khasminskii} about the continuity of stochastic convolution terms.

Assume Assumptions  \ref{assum:elliptic-operators}, and \ref{assum:noise}. Let $E_T$ be defined by \eqref{eq:E_T-def}. For arbitrary $\sigma \in E_T$, define the multiplication operators $R_n: [0,T]\to \mathscr{L}( L^2(\mathcal{O}))$ such that for any $n \in \{1,...r\}$, $t \in [0,T]$, and $\xi \in \mathcal{O}$, and $f \in L^2(\mathcal{O})$.
\begin{equation} \label{eq:Rn}
  [R_n(t) f](\xi): = \sigma_n(t,\xi)f(\xi).
\end{equation}
In this appendix, we investigate the continuity in time and space of the stochastic convolutions
\begin{equation} \label{eq:Z-mild-def-appendix}
  Z_i(t): = \int_0^t S_i(t-s)\sum_{n=1}^rR_{n}(s)dw_n(s)
\end{equation}
solving
\[dZ_i(t) = A_i Z_i(t) + \sum_{n=1}^r R_{n}(t)dw_n(t).\]
In these expressions, $S_i(t)$ are the semigroups defined in Section \ref{SS:semigroup} generated by the unbounded operators $A_i$ and $w_n$ are Gaussian noises satisfying Assumption \ref{assum:noise}.
$Z_i(t)$ also solves
\[dZ_i(t) = B_i Z_i(t) + L_i Z_i(t) + \sum_{n=1}^r R_{n}(t)dw_n(t),\]
where $B_i$ and $L_i$ are defined in Proposition \ref{prop:A_i}.
The mild solution of $Z_i$ solves
\begin{equation} \label{eq:Z-mild}
  Z_i(t) = \int_0^t T_i(t-s) L_i Z_i(s)ds + \tilde{Z}_i(t)
\end{equation}
where
\begin{equation} \label{eq:Z-stoch-conv-appendix}
  \tilde{Z}_i(t):= \int_0^t T_i(t-s) \sum_{n=1}^rR_{n}(s)dw_n(s)
\end{equation}
where $T_i(t)$ is the semigroup generated by the realization of $\mathcal{B}_i$ in $L^2(\mathcal{O})$ where $\mathcal{A}$ satisfies Assumption \ref{assum:elliptic-operators} and $w$ satisfies Assumption \ref{assum:noise}.
Similar results can be found in Section 4 of \cite{c-2003}.
For $i \in \{1,...,r\}$, let $T_i(t)$ be the semigroup on $H=L^2(\mathcal{O})$ generated by $B_i$ (see Proposition \ref{prop:A_i}).
There exists a kernel
\begin{equation} \label{eq:kernel-def}
  K_i(t,\xi,\eta): = \sum_{k=1}^\infty e^{-\alpha_{i,k} t} e_{i,k}(\xi)e_{i,k}(\eta)
\end{equation}
such that for any $\varphi \in L^2(\mathcal{O})$ and $t>0$,
\begin{equation} \label{eq:kernel}
  [T_i(t) \varphi](\xi) = \int_\mathcal{O} K_i(t,\xi,\eta)\varphi(\eta)d\eta.
\end{equation}

The $T_i(t)$  semigroups have many useful smoothing properties including for $\gamma \in (0,1)$, $p>1$,
\begin{align}
  &|T_i(t) \varphi|_{W^{\gamma,p}(\mathcal{O})} \leq C t^{-\frac{\gamma}{2}} |\varphi|_{L^p(\mathcal{O})} \label{eq:semigroup-regularization-Sobolev}\\
  &|T_i(t) \varphi|_{C(\bar {\mathcal{O}})}  \leq C t^{-\frac{1}{2}} |\varphi|_{C^{-1}(\bar {\mathcal{O}})} \label{eq:semigroup-regularization-C-1}
\end{align}

We use the stochastic factorization method of Da Prato and Zabczyk \cite{dpz}. For $\alpha \in \left( 0, \frac{1}{2} \right)$, let
\begin{equation} \label{eq:Z-alpha-appendix}
  \tilde{Z}_{i,\alpha}(\tau):= \int_0^\tau (\tau-s)^{-\alpha}T_i(\tau-s)\sum_{n=1}^r R_{n}(s)dw_n(s).
\end{equation}
Then because $\int_s^t (t-\tau)^{\alpha-1}(\tau-s)^{-\alpha}d\tau = \frac{\pi}{\sin(\pi \alpha)}$,
\begin{equation} \label{eq:factorization}
  \tilde{Z}_i(t) = \frac{\sin(\pi \alpha)}{\pi}\int_0^t (t-\tau)^{1-\alpha}\tilde{Z}_{i,\alpha}(\tau)d\tau.
\end{equation}

\begin{lemma} \label{lem:sum-of-semigroup}
  Let  $\sigma \in E$. Let $R_n \in \mathscr{L}(L^2(\mathcal{O}))$ be given by $[R_n h](\xi) = \sigma(\xi)h(\xi)$.
   Let $\{f_{n,j}\}$ and $\beta>0$ be as in Assumption \ref{assum:noise}.
  Then for any $i \in \{1,...,r\}$, $t>0$, and $\xi \in \mathcal{O}$,
  \begin{equation} \label{eq:sum-of-squares}
    \sum_{j=1}^\infty  \left(\left(T_i(t) \sum_{n=1}^r R_n f_{n,j} \right)(\xi) \right)^2
    \leq C t^{-\beta}  |\sigma|_{E}^2.
  \end{equation}
\end{lemma}
\begin{proof}
  Using the Kernel representation of the semigroup \eqref{eq:kernel},
  \begin{align*}
    &\sum_{j=1}^\infty  \left(\left(T_i(t) \sum_{n=1}^rR_n f_{n,j} \right)(\xi) \right)^2\\
    &\leq\sum_{j=1}^\infty  \left(  \int_\mathcal{O} \sum_{n=1}^r K_i(t,\xi,\eta) \sigma_n(\eta) f_{n,j}(\eta) d\eta \right)^2\\
    &\leq C\sum_{j=1}^\infty \sum_{n=1}^r \left(  \int_\mathcal{O} K_i(t,\xi,\eta) \sigma_n(\eta) f_{n,j}(\eta) d\eta \right)^2.
  \end{align*}
  Because, for fixed $n$, $\{f_{n,j}\}_{n=1}^\infty $ is a complete orthonormal basis of $L^2(\mathcal{O})$,
  \begin{align} \label{eq:kernel-squared}
    &\sum_{n=1}^r\sum_{j=1}^\infty  \left( \int_\mathcal{O} K_i(t,\xi,\eta) \sigma_n(\eta) f_{n,j}(\eta) d\eta \right)^2\nonumber\\
    &\leq  \sum_{n=1}^r\int_\mathcal{O} \left(  K_i(t,\xi,\eta) \sigma_n(\eta) \right)^2d\eta\nonumber\\
    &\leq  C |\sigma|_{E}^2 \int_\mathcal{O} \left( K_i(t,\xi,\eta) \right)^2d\eta.
  \end{align}
  Because $\{e_{i,k}\}_{k=1}^\infty$ are an orthonormal basis of eigenfunctions, by \eqref{eq:kernel-def}
  \begin{align*}
    &\int_\mathcal{O} \left( K_i(t,\xi,\eta) \right)^2d\eta\\
    &\leq \int_\mathcal{O} \left( \sum_{k=1}^\infty e^{-\alpha_{i,k} t} e_{i,k}(\eta) e_{i,k}(\xi) \right)^2 d\eta\\
    &\leq \sum_{k=1}^\infty e^{-2\alpha_{i,k} t} |e_{i,k}(\xi)|^2
  \end{align*}
  Let $c_\beta := \sup_{x>0} x^\beta e^{-x}<+\infty$.
  Then.
  \begin{align*}
    &\sum_{k=1}^\infty e^{-2\alpha_{i,k} t} |e_{i,k}(\xi)|^2\\
    & \leq \sum_{k=1}^\infty 2^{-\beta}\alpha_{i,k}^{-\beta}|e_{i,k}|_{L^\infty(\mathcal{O})}^2 t^{-\beta} (2\alpha_{i,k} t)^{\beta} e^{-2\alpha_{i,k} t}\\
    & \leq c_\beta t^{-\beta} \sum_{k=1}^\infty \alpha_{i,k}^{-\beta} |e_{i,k}|_{L^\infty(\mathcal{O})}^2\\
    &\leq C t^{-\beta}.
  \end{align*}
  The sum is finite by Assumption \ref{assum:noise}. The result now follows from \eqref{eq:kernel-squared}.
\end{proof}

\begin{lemma}
  For each $n \in \{1,...,r\}$, let $\{f_{n,j}\}_{j=1}^\infty$ be the complete orthonormal basis of $L^2(\mathcal{O})$ and let $\{\lambda_{n,j}\}_{j=1}^\infty$ be the eigenvalues from Assumption \ref{assum:noise}. Let $\beta$ and $\rho$ be the constants from Assumption \ref{assum:noise}.  There exists $C>0$ such that for any  $\sigma \in E$, $R_n$ defined as in Lemma \ref{lem:sum-of-semigroup},  $i \in \{1,...,r\}$, $t>0$, and $\xi \in \mathcal{O}$,
  \begin{equation} \label{eq:sum-of-squares-w-lambda}
    \sum_{j=1}^\infty  \left(\left(T_i(t) \sum_{n=1}^rR_n \lambda_{n,j} f_{n,j} \right)(\xi) \right)^2
    \leq C t^{-\frac{\beta(\rho-2)}{\rho}}  |\sigma|_{E}^2.
  \end{equation}
\end{lemma}
\begin{proof}
  By the H\"older inequality with exponents $\frac{\rho}{2}$ and $\frac{\rho}{\rho-2}$,
  \begin{align*}
    &\sum_{j=1}^\infty  \left(\left(T_i(t) \sum_{n=1}^rR_n \lambda_{n,j} f_{n,j} \right)(\xi) \right)^2\\
    &\leq C \sum_{j=1}^\infty \sum_{n=1}^r \left(\left(T_i(t) R_n \lambda_{n,j} f_{n,j} \right)(\xi) \right)^2\\
    &\leq C\left(\sum_{j=1}^\infty \sum_{n=1}^r \lambda_{n,j}^\rho |(T_i(t)R_{n} f_{n,j})(\xi)|^2 \right)^{\frac{2}{\rho}}
    \left(\sum_{j=1}^\infty \sum_{n=1}^r |(T_i(t)R_{n} f_{n,j})(\xi)|^2  \right)^{\frac{\rho-2}{\rho}}
  \end{align*}
  Because $T_i(t)$ is a contraction semigroup on $\tilde{E}$,
  \begin{equation}
    |(T_i(t)R_{n}f_{n,j})(\xi)| \leq \sup_{\eta \in \mathcal{O}} |\sigma_{n}(\eta)f_{n,j}(\eta)| \leq |\sigma|_{E}|f_{n,j}|_{L^\infty(\mathcal{O})}.
  \end{equation}
  Therefore,
  \begin{align*}
    &\sum_{j=1}^\infty \sum_{n=1}^r \lambda_{n,j}^\rho |(T_i(t)\sigma_{n} f_{n,j})(\xi)|^2  \leq \left( \sum_{j=1}^\infty \sum_{n=1}^r \lambda_{n,j}^\rho |f_{n,j}|_{L^\infty(\mathcal{O})}^2  \right) |\sigma|_{E}^2.
  \end{align*}
  The summation in the above display is finite by Assumption \ref{assum:noise}. By also applying \eqref{eq:sum-of-squares},
  \begin{align*}
    &\left(\sum_{j=1}^\infty \sum_{n=1}^r \lambda_{n,j}^\rho |(T_i(t)\sigma_{n} f_{n,j})(\xi)|^2 \right)^{\frac{2}{\rho}}
    \left(\sum_{j=1}^\infty |(T_i(t)\sigma_{n} f_{n,j})(\xi)|^2  \right)^{\frac{\rho-2}{\rho}} \\
    &\leq Ct^{-\frac{\beta(\rho-2)}{\rho}} |\sigma|_{E}^2.
  \end{align*}
\end{proof}

\begin{lemma}[Estimates on $\tilde{Z}_{i,\alpha}$] \label{lem:Z_alpha-bound}
  Let $\sigma \in E_T$ and let $\tilde{Z}_{i,\alpha}$ be given by \eqref{eq:Z-alpha-appendix}. For any $p>1$ there exists a constant $C>0$ such that for any $i \in \{1,...,r\}$, $t>0$, $\xi \in D$ ,
  \begin{equation} \label{eq:Z_alpha-bound}
    \E |\tilde{Z}_{i,\alpha}(t,\xi)|^p \leq C \E \left(\int_0^t (t-s)^{-2\alpha - \frac{\beta(\rho-2)}{\rho}} |\sigma(s)|_{E}^2ds\right)^{\frac{p}{2}}
  \end{equation}
\end{lemma}
\begin{proof}
  By \eqref{eq:noise} and the BDG inequality, for any $\xi \in \mathcal{O}$, $t>0$,
  \begin{align*}
    &\E|\tilde{Z}_{i,\alpha}(t,\xi)|^p\\
    &\leq C \E \left( \sum_{j=1}^\infty  \int_0^t (t-s)^{-2\alpha} \left|\left(T_i(t-s)\sum_{n=1}^rR_{n}(s) \lambda_{n,j} f_{n,j}\right)(\xi)\right|^2 ds\right)^{\frac{p}{2}}.\\
  \end{align*}
  By \eqref{eq:sum-of-squares-w-lambda},
  \begin{align*}
    &\E|\tilde{Z}_{i,\alpha}(t,\xi)|^p \leq C \E \left(\int_0^t (t-s)^{-2\alpha - \frac{\beta(\rho-2)}{\rho}} |\sigma(s)|_{E}^2ds\right)^{\frac{p}{2}}.
  \end{align*}
\end{proof}
\begin{theorem}[Bounds on $\tilde{Z}_i$] \label{thm:tilde-Z-bound}
  Let $\tilde{Z}_i$ be given by \eqref{eq:Z-stoch-conv-appendix}. For $\alpha \in \left( 0, \frac{1}{2}\left(1- \frac{\beta(\rho-2)}{\rho} \right) \right)$, $\gamma \in (0,\alpha)$,  $p> \max\left\{\frac{1}{\alpha - \gamma}, \frac{d}{\gamma} \right\}$, and $T>0$, there exists $C=C(\alpha,\gamma,p,T)>0$ such that
  \begin{align} \label{eq:tilde-Z-bound}
    &\E \sup_{t \in [0,T]} \sup_{\xi \in \mathcal{O}} |\tilde{Z}_i(t,\xi)|^p  \leq C \E \int_0^T \left(\int_0^t (t-s)^{-2\alpha - \frac{\beta(\rho-2)}{\rho}}  |\sigma(s)|_{E}^2 ds \right)^{\frac{p}{2}}dt.
  \end{align}
\end{theorem}

\begin{proof}
  The fractional Sobolev space $W^{\gamma,p}(\mathcal{O})$ embeds continuously into $\tilde{E}$ whenever $\gamma \in (0,1)$ and $\gamma p >d$ \cite{sobolev}. Therefore, by factorization \eqref{eq:factorization}, the fractional Sobolev embedding, and the regularization of the $T_i$ semigroups \eqref{eq:semigroup-regularization-Sobolev},
  \begin{align*}
    &\E \sup_{t \in [0,T]} \sup_{\xi \in \mathcal{O}} |\tilde{Z}_i(t,\xi)|^p\\
    &\leq C \E \sup_{t \in [0,T]} |\tilde{Z}_i(t)|_{W^{\gamma,p}(\mathcal{O})}^p\\
    &\leq C \E \sup_{t \in [0,T]} \left| \int_0^t (t-s)^{\alpha -1}T_i(t-s) \tilde{Z}_{i,\alpha}(s)ds \right|_{W^{\gamma,p}(\mathcal{O})}^p\\
    &\leq C \E \sup_{t \in [0,T]} \left(\int_0^t (t-s)^{\alpha - 1 -\gamma} |\tilde{Z}_{i,\alpha}(s)|_{L^p(\mathcal{O})}ds\right)^p.
  \end{align*}
  By the H\"older inequality,
  \begin{align*}
    &\leq C \left(\int_0^T s^{\frac{(\alpha - 1 - \gamma)p}{p-1}}ds \right)^{p-1} \E \int_0^T |\tilde{Z}_{i,\alpha}(t)|_{L^p(\mathcal{O})}^pdt.
  \end{align*}
  By the fact that $p> \frac{1}{\alpha - \gamma}$, the first integral on the right-hand side is finite. By Lemma \ref{lem:Z_alpha-bound}, \eqref{eq:tilde-Z-bound} follows.
\end{proof}

\begin{theorem} \label{thm:Z-bound}
  Let $Z_i$ be given by \eqref{eq:Z-mild-def-appendix}. For $\alpha \in \left( 0, \frac{1}{2}\left(1- \frac{\beta(\rho-2)}{\rho} \right) \right)$, $\gamma \in (0,\alpha)$,  $p> \max\left\{\frac{1}{\alpha - \gamma}, \frac{d}{\gamma} \right\}$, and $T>0$, there exists $C=C(\alpha,\gamma,p,T)>0$ such that
  \begin{equation} \label{eq:Z-bound}
    \E \sup_{t \in [0,T]} \sup_{\xi \in \mathcal{O}} |Z_i(t,\xi)|^p \leq C \E \int_0^T \left(\int_0^t (t-s)^{-2\alpha - \frac{\beta(\rho-2)}{\rho}} |\sigma(s)|_{E}^2 ds \right)^{\frac{p}{2}}dt
  \end{equation}
\end{theorem}

\begin{proof}
  By \eqref{eq:Z-mild} and \eqref{eq:semigroup-regularization-C-1},
  \begin{align*}
    &|Z_i(t)|_{\tilde{E}}^p \leq \left(\int_0^t |T_i(t-s) L_i Z_i(s)|_{\tilde{E}}ds\right)^p + C|\tilde{Z}_i(t)|_{\tilde{E}}^p\\
    &\leq C\left(\int_0^t (t-s)^{-\frac{1}{2}} |Z_i(s)|_{\tilde{E}}ds\right)^p + C|\tilde{Z}_i(t)|_{\tilde{E}}^p.
  \end{align*}
  By the H\"older inequality for $p>2$,
  \begin{align*}
    &|Z_i(t)|_{\tilde{E}}^p \leq  C\left(\int_0^t (t-s)^{-\frac{p}{2(p-1)}}ds \right)^{p-1} \int_0^t |Z_i(s)|_{\tilde{E}}^pds + C|\tilde{Z}_i(t)|_{\tilde{E}}^p.
  \end{align*}
  Taking expectation and applying Gr\"onwall's inequality,
  \[\E\sup_{t \in [0,T]} |Z_i(t)|^p_{\tilde{E}} \leq C_{T,p} \sup_{t \in [0,T]}|\tilde{Z}_i(t)|_{\tilde{E}}.\]
  The result follows from \eqref{eq:tilde-Z-bound}.
\end{proof}

We finish this section with the analysis of similar terms where the stochastic noise has been replaced by a $L^2([0,T]\times \mathcal{O}\times \{1,...,r\})$ control.

\begin{theorem} \label{thm:control-Y-bounds}
  Let $u = (u_1,...,u_r) \in L^2([0,T]\times \mathcal{O}\times \{1,...,r\})$. Let $\sigma \in  E_T$ and let $R$ be given by \eqref{eq:Rn}. Let $Y^u_i$ (weakly) solve
  \[dY^u_i(t)=   [A_i Y^u_i(t) + \sum_{n=1}^r R_n(t)Q_nu_n(t)]dt.\]
  Then $Y^u_i \in C([0,T]\times \mathcal{O})$ and there exists $C>0$, independent of $u$ and $\sigma$ such that
  \begin{align} \label{eq:Y-u-bounds}
    &\sup_{t \in [0,T]} \sup_{\xi \in \mathcal{O}}|Y^u_i(t,\xi)| \nonumber\\
    &\leq
    CN^{\frac{1}{2}} \sup_{t \in [0, T]}\left(\int_0^t (t-s)^{-\frac{\beta(\rho-2)}{\rho}}  |\sigma(s)|_E^2 ds  \right)^{\frac{1}{2}}.
  \end{align}
  where
  \[N = |u|^2_{L^2([0,T]\times \mathcal{O}\times \{1,...,r\})} = \sum_{n=1}^r \int_0^T \int_{\mathcal{O}} |u_n(s,\xi)|^2d\xi ds \]
\end{theorem}

\begin{proof}
  As we did for the stochastic term, let
  \begin{equation} \label{eq:tilde-Y}
    \tilde{Y}^u_i(t) = \int_0^t T_i(t-s) \sum_{n=1}^r R_n(s) Q_n u_n(s)ds.
  \end{equation}
  This is the solution to
  \[d \tilde{Y}^u_i(t) = [B_i \tilde{Y}^u_i(t) + \sum_{n=1}^r R_n(t)Q_n u_n(t)]dt.\]
  We can rewrite \eqref{eq:tilde-Y} as
  \[\tilde{Y}^u_i(t) = \sum_{n=1}^r \sum_{j=1}^\infty \int_0^t T_i(t-s)  R_n(s) \lambda_{n,j} f_{n,j} \left<u_n(s),f_{n,j} \right>_{L^2(\mathcal{O})}ds.\]
  By the H\"older inequality, for any $\xi \in \mathcal{O}$,
  \begin{align*}
    |\tilde{Y}^u_i(t,\xi)| \leq &\left( \sum_{n=1}^r \sum_{j=1}^\infty \int_0^t |(T_i(t-s)  R_n(s) \lambda_{n,j} f_{n,j})(\xi)|^2 ds \right)^{\frac{1}{2}}\\
     &\qquad\times\left(\sum_{n=1}^r \sum_{j=1}^\infty \int_0^t \left<u_n(s),f_{n,j} \right>_{L^2(\mathcal{O})}^2ds \right)^{\frac{1}{2}}.
  \end{align*}
  By \eqref{eq:sum-of-squares-w-lambda} and the fact that $\{f_{n,j}\}_{j=1}^\infty$ is a complete orthonormal basis of $L^2(\mathcal{O})$, for any $t >0$ and $\xi \in \mathcal{O}$,
  \begin{align*}
    |\tilde{Y}^u_i(t,\xi)| \leq C &\left(\int_0^t (t-s)^{-\frac{\beta(\rho-2)}{\rho}}|\sigma(s)|_E^2 ds \right)^{\frac{1}{2}}
    \left(\sum_{n=1}^r \int_0^t \int_\mathcal{O} |u_n(s,\eta)|^2 d\eta ds \right)^{\frac{1}{2}}.
  \end{align*}

  The continuity of $\tilde{Y}^u_i$ in space in time can be shown by standard arguments (see, for example, \cite{Pazy}).

  Then $Y_i^u(t)$ solves
  \[Y^i_u(t) = \int_0^tT_i(t-s) L_i Y^u_i(s)ds + \tilde{Y}_i^u(t).\]
  By \eqref{eq:semigroup-regularization-C-1},
  \[\left|Y^u_i(t)\right|_{\tilde{E}} \leq \int_0^t (t-s)^{-\frac{1}{2}}\left|Y^u_i(s)\right|_{\tilde{E}}ds + |\tilde{Y}^u_i(t)|_{\tilde{E}}.\]
  The result follows by Gr\"onwall's inequality.
\end{proof}

\section{Bounds on the fixed-point mapping that are uniform with respect to the initial condition.} \label{S:unif-bounds}
Let $E$ and $E_T$ be defined as in Section \ref{S:notation-assumptions}. For any $z \in E_T$ with $z(0)=0$ and $x \in E$ let $\mathcal{M}_x(z)$ be the solution to the fixed-point problem
\begin{equation} \label{eq:M_x-def}
  \mathcal{M}_x(z)(t) := \mathcal{M} \left( S(\cdot)x + z \right)
\end{equation}
where $\mathcal{M}$ is the fixed point mapping defined in \eqref{eq:M-def} and $S$ is the semigroup defined in Section \ref{SS:semigroup}.
The following result establishes bounds on $\mathcal{M}_x$ that are independent of $x$ when the vector field $f$ features super-linear dissipativity (see Assumption \ref{assum:super-dissip}).

\begin{theorem} \label{thm:M-bound-super-dissip}
  Assume Assumptions \ref{assum:vector-field},  \ref{assum:elliptic-operators},  and \ref{assum:super-dissip}. For any $t>0$ there exists $C_t>0$ ($C_t$ increases as $t$ increases), independent of $x \in E$, such that for any $z \in E_t$ with $z(0)=0$ and any $x \in E$,
  \begin{equation} \label{eq:M-bound-super-dissip}
    |\mathcal{M}_x(z)(t)|_E \leq  C_t \left( 1+    t^{-\frac{1}{m-1}}  +|z|_{E_t}\right).
  \end{equation}
\end{theorem}
\begin{proof}
  Let
  \begin{equation}
    v(t) = \mathcal{M}_x(z)(t) - z(t).
  \end{equation}
  $v(t)$ solves the integral equation
  \begin{equation}
     v(t) = S(t)x + \int_0^t S(t-s)F(v(s) + z(s))ds.
  \end{equation}
  Therefore,  $v$ is weakly differentiable and solves
  \[\frac{\partial }{\partial t} v_i(t,\xi) = \mathcal{A}_i v_i(t,\xi) + g_i\big(v_i(t,\xi)+ z_i(t,\xi)\big) + h_i\big(v(t,\xi) + z(t,\xi)\big).\]
  Using the arguments of  Theorem 7.7 in \cite{dpz} and Proposition 6.2.2 of \cite{cerrai} we may assume without loss of generality that $v$ is strongly differentiable.

  By Proposition \ref{prop:left-deriv}, and Assumption \ref{assum:super-dissip}, for any $i_t \in \{1,...,r\}$ and  $\xi_t \in \mathcal{O}$ such that
  \begin{equation} \label{eq:v-maximizers}
    |v(t)|_E = |v_{i_t}(t,\xi_t)|
  \end{equation}
  \begin{align*}
    \frac{d^-}{dt}& |v(t)|_E \\
    \leq &\mathcal{A}_{i_t} v_{i_t}(t,x_t)\sgn(v_{i_t}(t,\xi_t))\\
    &+ g_{i_t}(t,\xi_t,v_{i_t}(t,\xi_t) + z_{i_t}(t,\xi_t))\sgn(v_{i_t}(t,\xi_t))\\
    &+ h_{i_t}(t,\xi, v(t,\xi_t) + z(t,\xi_t)) \sgn(v_{i_t}(t,\xi_t)).
  \end{align*}

  Because $(i_t,\xi_t)$ is a maximizer and $\mathcal{A}_{i_t}$ is an elliptic operator, the concavity of a function at is maximum/minimum implies that
  \begin{equation}
    \mathcal{A}_{i_t} v_{i_t}(t,x_t)\sgn(v_{i_t}(t,\xi_t)) \leq 0.
  \end{equation}

  We will estimate this derivative when $|v(t)|_E$ is large. If  $|v(t)|_E \geq 2 c_0 + 2 | z|_{E_t}$, where $c_0$ is from Assumption \ref{assum:super-dissip},
  then
  \begin{equation} \label{eq:v+z-ineq}
    |v_{i_t}(t,\xi_t) + z_{i_t}(t,\xi_t)| \geq |v_{i_t}(t,\xi_t)| - |z_{i_t}(t,\xi_t)| \geq  \frac{1}{2} | v_{i_t}(t,\xi_t)| > c_0
  \end{equation}
  and
  \begin{equation} \label{eq:equal-sgn}
    \sgn(v_{i_t}(t,\xi_t) + z_{i_t}(t,\xi_t)) = \sgn(v_{i_t}(t,\xi_t)).
  \end{equation}
  Therefore, \eqref{eq:g-super-dissip}, \eqref{eq:v+z-ineq}, and \eqref{eq:equal-sgn} guarantee that
  \begin{equation}
    g_{i_t}(t,\xi_t,v_{i_t}(t,\xi_t) + z_{i_t}(t,\xi_t))\sgn(v_{i_t}(t,\xi_t))
    \leq -\frac{\mu}{2^m} |v_{i_t}(t,\xi_t)|^m = -\frac{\mu}{2^m} |v(t)|_E^m.
  \end{equation}
  By \eqref{eq:h-lin-growth} in Assumption \ref{assum:vector-field}, for any $t \in [0,T]$,
  \begin{align}
    h_{i_t}(t,\xi, v(t,\xi_t) + z(t,\xi_t)) \sgn(v_{i_t}(t,\xi_t)) &\leq L(t)(1 + |v(t) + z(t)|_E) \nonumber\\
    &\leq L(t) \left(1 + \frac{3}{2} |v(t)|_E\right).
  \end{align}
  There exists a constant $C_T>c_0$, depending on $\mu$, $c_0$, and $L(T)$ such that whenever $|v(t)|_E> C_T + 2|z|_{E_T}$,
  \begin{equation}
    - \frac{\mu}{2^m}|v(t)|_E^m + L(T)\left(1 + \frac{3}{2} |v(t)|_E\right)\leq - \frac{\mu}{2^{m+1}}|v(t)|_E^m.
  \end{equation}
  Therefore, whenever $t \in [0,T]$ and
  \begin{equation}
    |v(t)|_E \geq C_T + 2|z|_{E_T}
  \end{equation}
  it follows from the above estimates that
  \begin{align} \label{eq:v-ode-super-poly}
    \frac{d^-}{dt} |v(t)|_E \leq - \frac{\mu}{2^{m+1}} |v(t)|_E^m.
  \end{align}

  We now separate this analysis into two cases: $|v(0)|_E \leq C_T + 2|z|_{E_T}$ and $|v(0)|_E > C_T + 2|z|_{E_T}$.

  If $|v(0)|_E \leq C_T + 2|z|_{E_T}$, then $|v(t)|_{E} \leq C_T + 2|z|_{E_T}$ for all $t \in [0,T]$ because $\frac{d^-}{dt}|v(t)|_E <0$ when $|v(t)|_E = C_T + 2|z|_{E_T}$. The negative left-derivative implies that $|v(t)|_E$ cannot ever reach the value $C_T + 2 |z|_{E_T}$ if it starts below that value.

  On the other hand, when $|v(0)|_E > C_T + 2|z|_{E_T}$, let $\tau = \inf\{t \in [0,T]: |v(t)|_E \leq C_T + 2|z|_{E_T}\}$. Observe that \eqref{eq:v-ode-super-poly} holds for all $t \in [0,\tau]$. By a comparison principle, there exists $C>0$ such that for all $t \in [0,\tau]$
  \begin{equation}
    |v(t)|_E \leq C t^{-\frac{1}{m-1}} \text{ for all } t \in [0,\tau]
  \end{equation}
  uniformly with respect to initial data. Then for $t \in (\tau,T]$,
  \begin{equation}
    |v(t)|_E \leq C_T + 2|z|_{E_T}
  \end{equation}
  because $|v(t)|_E$ cannot exceed this value once it goes below it.

  These calculations show that, independent of initial data, $|v(t)|_E$ is bounded by
  \begin{equation}
    |v(t)|_E \leq \max \left\{C t^{-\frac{1}{m-1}}, C_T + 2|z|_{E_T} \right\} \text{ for all } t \in [0,T].
  \end{equation}
  
  Finally,
  \[|M_x(z)(t)|_E \leq |v(t)|_E + |z(t)|_E \leq C_t \left(1 + t^{-\frac{1}{m-1}} + |z|_{E_t} \right)\]

\end{proof}

\end{appendix}

\bibliographystyle{amsplain}
\bibliography{2020}
\end{document}